\documentclass[12pt, a4paper]{article}
\usepackage[dvipsnames]{xcolor}
\usepackage[latin1]{inputenc}
\usepackage[T1]{fontenc}

\usepackage{pdfsync}
\usepackage{hyperref}
\usepackage{graphicx}
\usepackage{enumerate} 

\usepackage{amsmath}
\usepackage{amsthm}
\usepackage{amssymb}
\usepackage{scrextend}

\usepackage{subfig}
\usepackage{float}
\usepackage{multirow}
\usepackage{bbold}
\usepackage{geometry}
\numberwithin{equation}{section}

\theoremstyle{plain}
\newtheorem{thm}{Theorem}[section]
\newtheorem{coro}[thm]{Corollary}
\newtheorem{prop}[thm]{Proposition}
\newtheorem{lem}[thm]{Lemma}
\newtheorem{defi}[thm]{Definition}

\theoremstyle{definition}

\theoremstyle{remark}
\newtheorem{rem}[thm]{Remark}

\newcommand{\Psym}{\mathcal{P}_{\textrm{sym}}\left(X^N\right)}
\newcommand{\Psymlambda}{\mathcal{P}_{\textrm{sym},\overline{\lambda}}\left(X^N\right)}
\newcommand{\conv}{\textrm{conv}}
\newcommand{\ext}{\textrm{ext}}
\newcommand{\PQN}{\mathcal{P}_{\frac{1}{N}}\left(X\right)}
\newcommand{\EsymN}{E_{\textrm{sym}}^N}

\newcommand{\Pcoef}{P_{\textrm{coef}}}
\newcommand{\PNrepk}{\mathcal{P}_{N\textrm{-rep}}\left(X^k\right)}
\newcommand{\PNreptwo}{\mathcal{P}_{N\textrm{-rep}}\left(X^2\right)}
\newcommand{\PNreplambda}{\mathcal{P}_{N\textrm{-rep},\overline{\lambda}}\left(X^2\right)}
\newcommand{\PsymMonge}{\mathcal{P}_{\textrm{sym,Monge}}\left(X^N\right)}
\newcommand{\PsymMongeConv}{\mathcal{P}_{\textrm{sym,Monge}}^{\textrm{conv}}\left(X^N\right)}
\newcommand{\PNrepMonge}{\mathcal{P}_{N\textrm{-rep,Monge}}\left(X^2\right)}
\newcommand{\PNrepMongeConv}{\mathcal{P}_{N\textrm{-rep,Monge}}^{\textrm{conv}}\left(X^2\right)}

\title{Geometry of Kantorovich Polytopes and Support of Optimizers for Repulsive Multi-Marginal Optimal Transport on Finite State Spaces}

\author{\normalsize
Daniela V\"{o}gler\thanks{Faculty of Mathematics, Technische Universit\"at M\"unchen,
\texttt{voegler@ma.tum.de}, +49-89-289-17900 (Tel.), +49-89-289-17932 (Fax)}
}
\date{}

\begin{document}
\maketitle
\vspace{-4mm}
\begin{addmargin}[2em]{2em}
\footnotesize
\noindent {\bf Abstract.} We consider symmetric multi-marginal Kantorovich optimal transport problems on finite state spaces with uniform-marginal constraint. These problems consist of minimizing a linear objective function over a high-dimensional polytope, here referred to as Kantorovich polytope. The presented results are of split nature, computational and theoretical. Within the computational part only small numbers of marginals $N$ and marginal sites $\ell$ are considered. This restriction allows us to computationally determine all extreme points of the Kantorovich polytope and investigate how many of them are in compliance with the in optimal transport typical Monge ansatz. Singling out the results for $\ell=3$ discretization points and pairwise symmetric cost functions enables us to visually compare Kantorovich's to Monge's ansatz space for a varying number of marginals. Finally we present a necessary support-condition for optimizers which is inspired by the insights the said model problem on three sites provided. This result is not limited to the case of $\ell=3$ sites and applies to symmetric pair-costs whose diagonal entries lie above a cost-specific threshold. In case $N$ and $\ell$ display certain relationships the discussed condition provides an optimizer in Monge-form and implies its uniqueness as a solution of the considered Kantorovich optimal transport problem.
\end{addmargin}
\vspace{4mm}
\noindent \textbf{Acknowledgments:} The author thanks Gero Friesecke and S\"oren Behr for helpful discussions.
\section{Introduction}
\label{sec:Intro}
\noindent In general multi-marginal Kantorovich optimal transport (OT) problems aim at coupling $N$ probability measures $\lambda^{(1)}, \ldots , \lambda^{(N)}$ optimally with respect to a given cost function $c$ (see \eqref{eq:ProblemN} for a discrete symmetric OT problem). These problems arise in various fields of research, ranging from economics \cite{CE10, CMN10} through mathematical finance \cite{BHP13, GHT14} and image processing \cite{AC11, RPDB12} to electronic structure \cite{CFK13, BDG12}.
\newline
\newline
Here we consider a symmetric multi-marginal Kantorovich OT problem on finite state spaces given by
\begin{multline}
\label{eq:ProblemN}
\textrm{Minimize } \int_{X^N} c(x_1,\ldots, x_N) d\gamma(x_1,\ldots,x_N) \\ \textrm{ over } \gamma \in \Psym \textrm{ subject to } \gamma \mapsto \overline{\lambda}.
\end{multline}
$X$ denotes a finite state space as defined in \eqref{eq:FiniteStateSpace}, $c : X^N \to \mathbb{R} \cup \{+\infty\}$ an arbitrary symmetric cost function, $\overline{\lambda}$ the uniform marginal as defined in \eqref{eq:UniMar} and $\Psym$ the set of symmetric probability measures on $X^N$, where a probability measure $\gamma$ on $X^N$ is symmetric if 
\begin{multline*}
\gamma \left( A_1 \times \dots \times A_N \right) = \gamma\left( A_{\sigma(1)} \times \dots \times A_{\sigma(N)} \right) \\ \textrm{ for all subsets } A_1,\ldots,A_N \textrm{ of } X \textrm{ and all permutations } \sigma.
\end{multline*} 
Any $\gamma \in \Psym$ fulfills $\gamma \mapsto \overline{\lambda}$ if and only if $\gamma$ has equal one-point marginals $\overline{\lambda}$, i.e., 
\begin{equation*}
\gamma\left( X^{k-1} \times A_k \times X^{N-k} \right) = \overline{\lambda}(A_k) \textrm{ for all subsets } A_k \textrm{ of } X \textrm{ and all } k \in \{1,\ldots,N\}.  
\end{equation*} 
Multi-marginal OT problems of form \eqref{eq:ProblemN} were already considered in \cite{FV18} and \cite{GF18}. 
While \cite{GF18} discusses the validity of Monge's approach in the setting of 3 marginals and 3 sites, \cite{FV18} introduces a sufficient ansatz space for problem \eqref{eq:ProblemN} (see Section \ref{sec:ClassKan} as well as Remark \ref{rem:IntObsTwo} for information about the content of these papers). The present paper accompanies these previous considerations. In particular, some of the used nomenclature and notation is already introduced there. 
\newline
\newline
For finite state spaces 
\begin{equation}
\label{eq:FiniteStateSpace}
X=\{a_1,\ldots,a_{\ell}\}
\end{equation}
consisting of $\ell$ distinct points $a_1,\ldots,a_{\ell}$, the uniform probability measure 
\begin{equation}
\label{eq:UniMar}
\overline{\lambda} = \sum_{i=1}^\ell \frac{1}{\ell} \delta_{a_i}
\end{equation}
on $X$ is the prototypical marginal. The corresponding multi-marginal Kantorovich OT problems, i.e., problems of form \eqref{eq:ProblemN} with $\mathcal{P}\left( X^{N} \right)$ replacing $\Psym$, appear directly as assignment problems (see \cite{Sp00, BDM12} for reviews) and arise from continuous problems via equi-mass discretization \cite{CFM14}. 
\newline
\newline
Note that the restriction to symmetric probability measures in problem \eqref{eq:ProblemN} is motivated by a physical application. Modeling the electronic structure of a molecule with $N$ electrons in a discretized setting, is a prototypical  application of multi-marginal OT on finite state spaces. In this context, $X$ corresponds to a set of $\ell$ discretization points in $\mathbb{R}^3$ and any coupling $\gamma$ of the $N$ marginals $\overline{\lambda}, \ldots, \overline{\lambda}$ describes a joint probability distribution regarding the electron positions in an $N$-electron molecule. Then the marginal condition ensures that each discretization point is occupied equally often and the cost function $c : X^N \to \mathbb{R} \cup \{ + \infty\}$ embodies the electron interaction energy. As electrons are indistinguishable the considered cost functions are usually symmetric, i.e., invariant under argument permutation. These symmetric cost functions are 'dual' to the set of symmetric probability measures on the product space $X^N$ in the following sense: There always exists an optimal coupling of $\overline{\lambda}, \ldots, \overline{\lambda}$ that is symmetric. 
\newline
\newline
The interaction energy between electrons displays a pairwise structure, i.e., $c(x_1,\ldots,x_N) = \sum_{1\leq i < j \leq N} v(x_i,x_j)$, with the Coulomb cost $\sum_{1\leq i < j \leq N} \frac{1}{|x_i-x_j|}$ being the prototypical example. Here $|\cdot|$ is the Euclidean norm in $\mathbb{R}^d$. As discussed in Section \ref{sec:IntroRed}, this pairwise structure allows us to reformulate the multi-marginal OT problem \eqref{eq:ProblemN} as 
\begin{equation}
\label{eq:ProblemTwo}
\textrm{Minimize } \int_{X^2} v(x,y)  d\mu(x,y) \textrm{ over } \mu \in \PNreptwo \textrm{ subject to } \mu \mapsto \overline{\lambda}.
\end{equation} 
This reformulated problem was initially introduced in \cite{FMPCK13}. In \eqref{eq:ProblemTwo}, $\PNreptwo \subseteq \mathcal{P}\left(X^2\right)$ can be interpreted as a 'reduced version' of $\Psym$. 
\newline
\newline
The goal of this paper is to present new insights into problem \eqref{eq:ProblemN} which arose out of a careful consideration of the polytope formed by the admissible trial states. Within this consideration the role Monge states take on in the said polytope is investigated.
\newline
\newline
One of the central questions in the theory of optimal transportation is: Under which assumptions exists an optimal coupling that is supported on a graph (over the first variable)? Such optimizers are then called Monge-solutions (see \eqref{eq:MongeOne} - \eqref{eq:MongeThree}). 
In the case of two marginals this question is well understood; the existence of Monge-solutions is always respectively under very general conditions guaranteed (see the renowned Birkhoff-von Neumann theorem \cite{Bi46, vN53} regarding finite state spaces respectively, e.g., \cite{Vi09} regarding continuous state spaces). 
For multiple marginals the understanding of this question does not reach the same generality. However, there are isolated examples for Monge- and non-Monge-solutions. For the former, see \cite{He02, Ca03, Pa11, CFK13, BDG12, CDD13} as well as the fundamental paper by Gangbo and \'{S}wi\k{e}ch \cite{GS98} for an interesting selection. For the latter, we refer the interested reader to \cite{CN08, Pa12, FMPCK13, Pa13, CFP15, MP17, GKR18, GF18} regarding continuous state spaces as well as to \cite{Cs70, Kr07, LL14, GF18} regarding finite state spaces.
\newline
\newline
In order to understand Monge's approach in the present setting, we first take a quick glance at the 'unsymmetrized' OT problem, i.e., problems of form \eqref{eq:ProblemN} with $\mathcal{P}\left(X^N\right)$ replacing $ \Psym$. Then, an optimal coupling $\gamma \in \mathcal{P}\left( X^{N} \right)$ of the $N$ marginals $\overline{\lambda},\ldots,\overline{\lambda}$ is a Monge-solution if 
\begin{equation}
\label{eq:UnsMonge}
\gamma = \sum_{\nu = 1}^\ell \frac{1}{\ell} \delta_{T_1(a_\nu)}\otimes \dots \otimes \delta_{T_N(a_\nu)}  \textrm{ for } N \textrm{ permutations } T_1, \ldots, T_N  : X \to X,
\end{equation}
where $T: X \to X$ is a permutation if there exists a permutation of indices $\tau : \{ 1, \ldots, \ell \} \to \{ 1, \ldots, \ell \}$ such that $T\left(a_\nu \right) = a_{\tau(\nu)}$ for all  $\nu \in \{ 1, \ldots, \ell \}$. Demanding that the $T_k$s are permutations ensures that $\gamma$ is indeed a coupling of $\overline{\lambda},\ldots,\overline{\lambda}$: $T: X \to X$ is a permutation if and only if it pushes the uniform measure forward to itself, i.e., $T_\#\overline{\lambda} = \overline{\lambda}$. Here for any probability measure $\lambda = \sum_{\nu=1}^\ell \lambda_\nu \delta_{a_\nu}$ on $X$ and any map $T : X \to X$ the push-forward $T_\#\lambda$ of $\lambda$ under $T$ is defined by  $T_\#\lambda = \sum_{\nu=1}^\ell \lambda_\nu \delta_{T(a_\nu)}$. One may choose $T_1 = \textrm{id}$, i.e., $T_1(a) = a$ for all $a \in X$,  by re-ordering the sum in \eqref{eq:UnsMonge}. 
\newline
Regarding \eqref{eq:ProblemN} an admissible trial state $\hat{\gamma}$ is referred to as a (symmetrized) Monge state if it is the symmetrization of a probability measure $\gamma$ of form \eqref{eq:UnsMonge}, i.e., 
\begin{equation}
\label{eq:MongeOne}
\hat{\gamma}=\sum_{\nu=1}^\ell \frac{1}{\ell} S\delta_{T_1(a_\nu)} \otimes \dots \otimes \delta_{T_N(a_\nu)}
\end{equation}
such that 
\begin{equation}
\label{eq:MongeTwo}
T_k \# \overline{\lambda}= \overline{\lambda} \textrm{ for all } k\in \{1,\ldots,N\},
\end{equation}
or equivalently,
\begin{equation}
\label{eq:MongeThree}
T_1,\ldots,T_N:X\to X \textrm{ are permutations.}
\end{equation}
Here $S$ denotes the linear symmetrization operator in $N$ variables as defined in \eqref{eq:SymOp}. Probability measures on $X^N$ of form \eqref{eq:MongeOne}-\eqref{eq:MongeThree} are also said to be of Monge-type or in Monge-form and restricting the minimization problem \eqref{eq:ProblemN} to such measures yields the corresponding Monge problem. 
\newline
\newline 
In Section \ref{sec:ClassKan} the 'Kantorovich-coupling-polytope', $\Psymlambda$, as defined in \eqref{eq:SetKanCou}, will be identified with the 'coefficient-polytope' $\Pcoef$, as introduced in \eqref{eq:Pcoef}. Monge states then correspond to integer elements of the latter of the two polytopes. Both, the reformulation as well as the identification of Monge states with integer coefficients are based on results in \cite{FV18}.
\newline
This different view on the set of admissible trial states of problem \eqref{eq:ProblemN} makes Monge's approach more accessible, in the sense of, it is easy to decide whether a given coefficient vector in $\Pcoef$ embodies a Monge state or not. This grasp of the Monge concept allows us to numerically quantify the insufficiency of Monge's ansatz which in the present setting is established in \cite{GF18}. In more detail, for small problem-parameters we determine all extremal coefficients and partition them into a Monge and a non-Monge class. The mere results of this classification are interpreted and complemented by 'small' theoretical results building upon the numerical ones. 
\newline
In Section \ref{sec:IntroRed} the same numerical analysis is performed under the additional assumption of the cost function $c$ in \eqref{eq:ProblemN} displaying pairwise structure.
\newline
\newline 
In Section \ref{sec:VisComp}, we will consider a model problem of optimally coupling the $N$ marginals $\overline{\lambda},\ldots, \overline{\lambda}$ with respect to a cost function of pairwise symmetric structure, where the finite state space $X$ consists only of three states, i.e., $X = \{a_1,a_2,a_3\}$. In particular $\overline{\lambda}$ is the uniform probability measure on $\{a_1, a_2, a_3\}$ and the domain of the cost function is given by $\{a_1, a_2, a_3\}^N$. The present setting allows us to draw a visual comparison between Kantorovich's and Monge's ansatz as depicted in Figure \ref{fig:KanMonge}. We further compare both OT approaches by volume of (the convex hull of) the respective set of admissible trial states and establish a computationally simple upper bound on the optimal value in \eqref{eq:ProblemTwo}. 
\begin{figure}[h]
\captionsetup[subfigure]{labelformat=empty}
\centering
\hspace{1mm}
\subfloat{\includegraphics[width=0.44\linewidth]{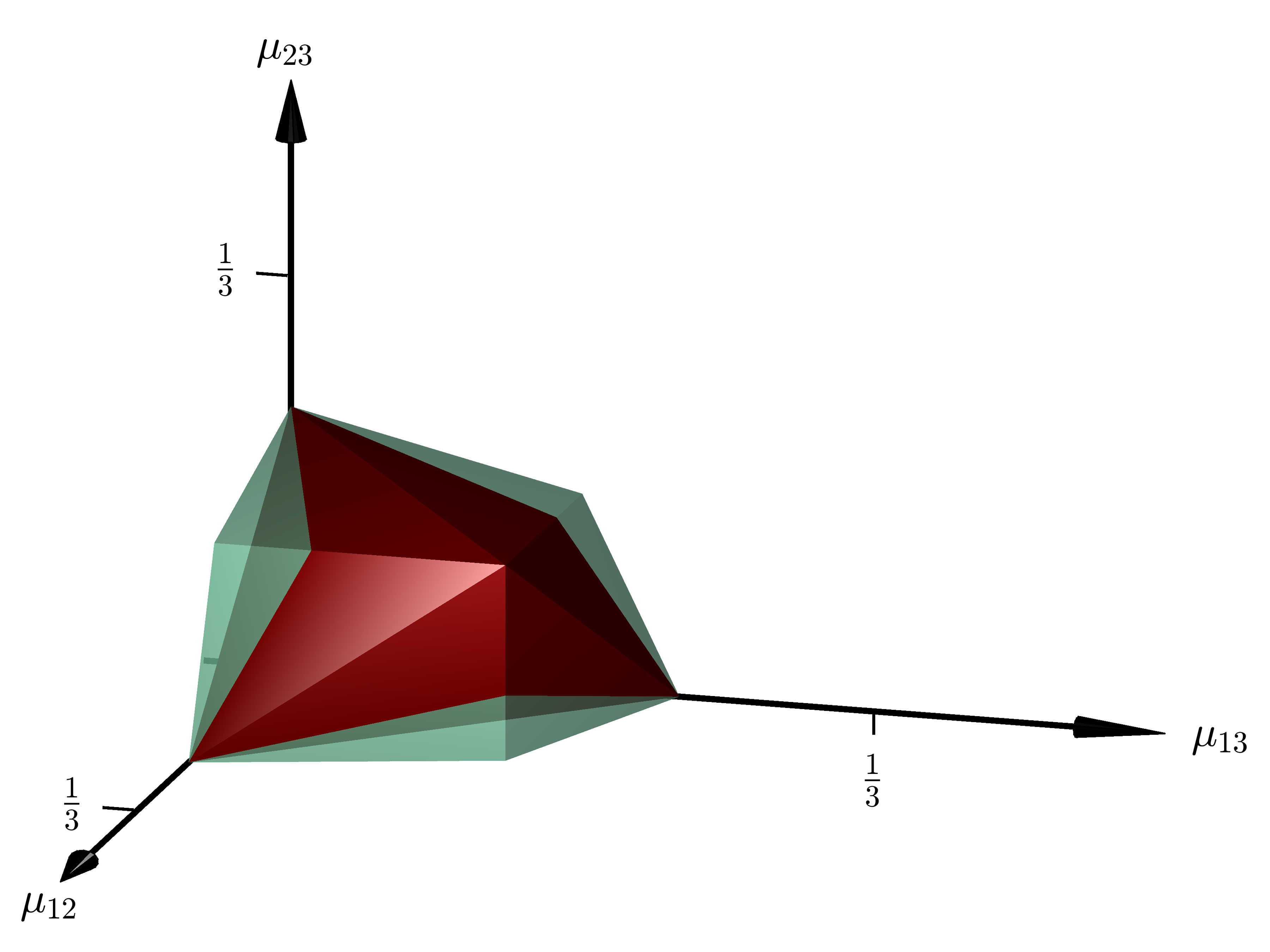}}
\hspace{1mm}
\subfloat{\includegraphics[width=0.44\linewidth]{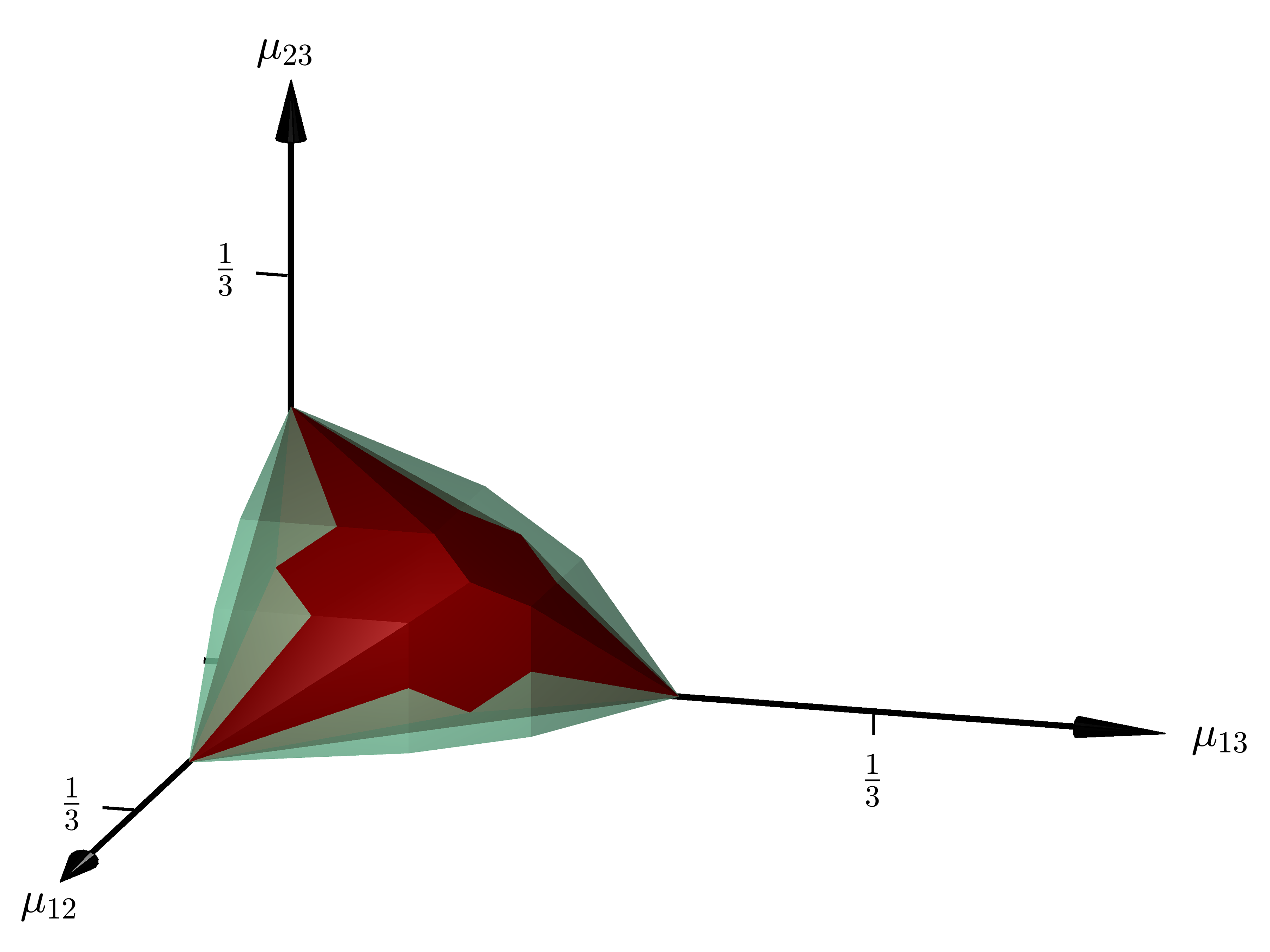}}
\hfill
\\
\hspace{1mm}
\subfloat{\includegraphics[width=0.44\linewidth]{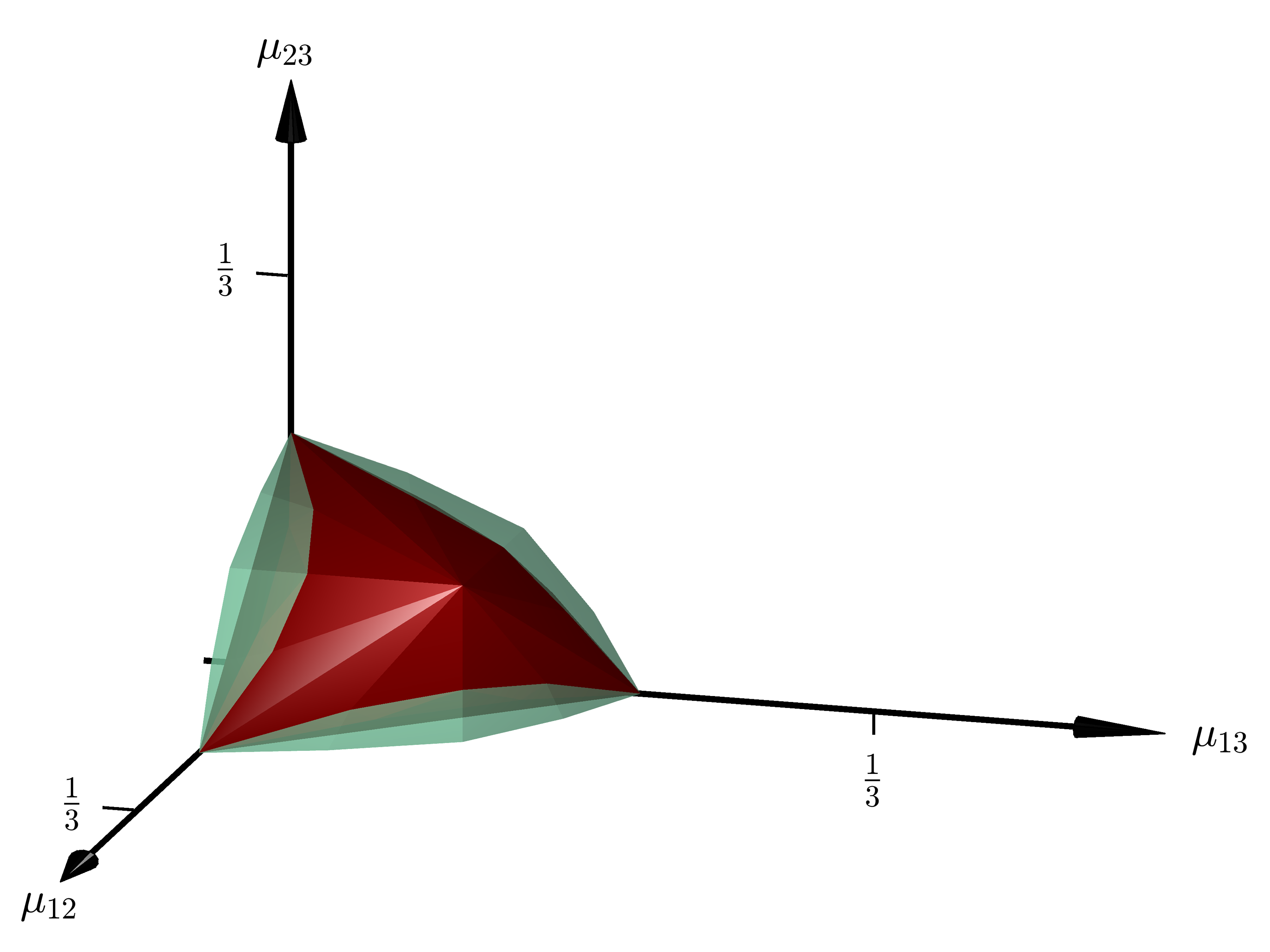}}
\hspace{1mm}
\subfloat{\includegraphics[width=0.44\linewidth]{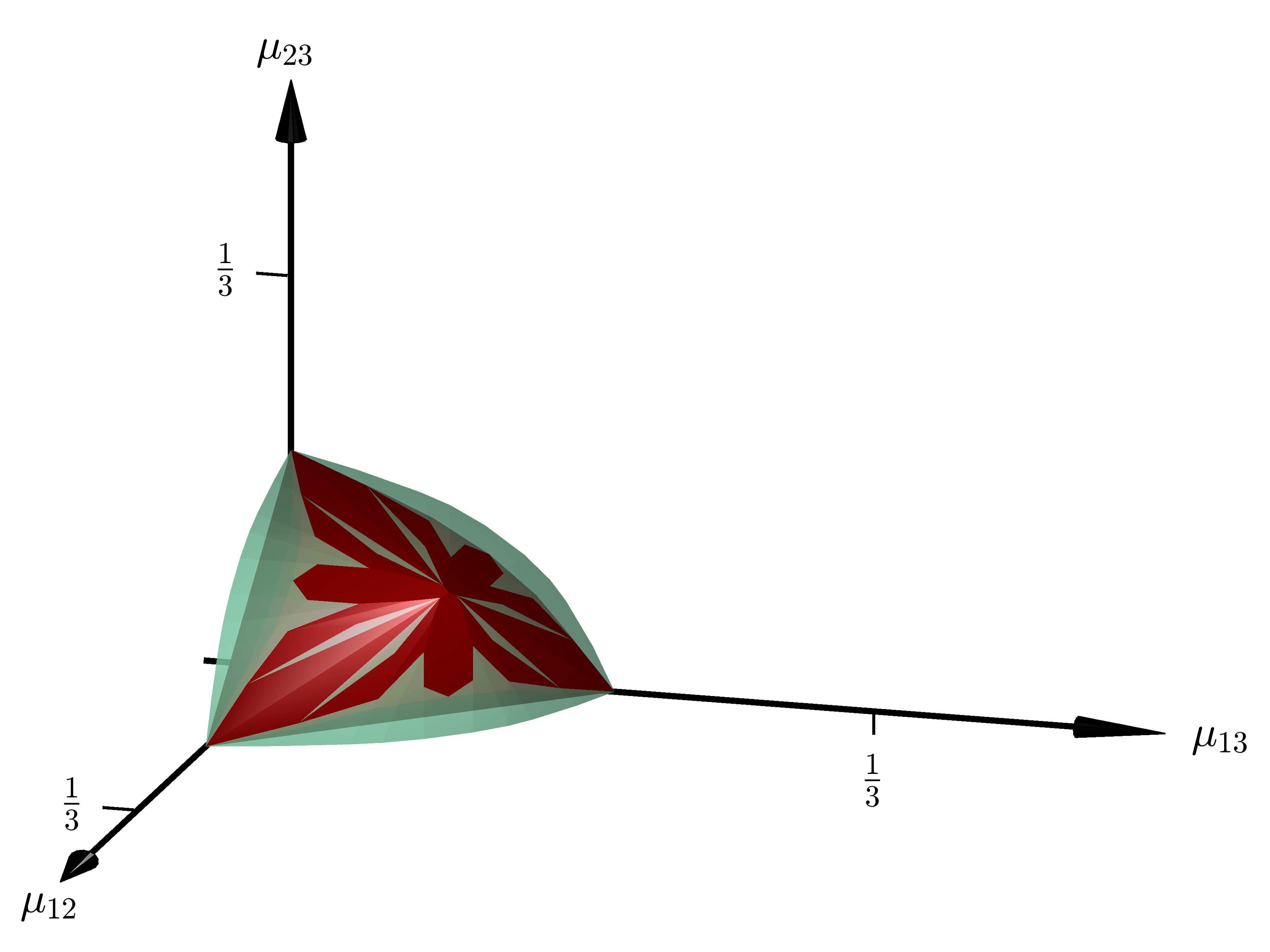}}
\hfill
\caption{The reduced Kantorovich (see \eqref{eq:DefPNreplambda}) respectively Monge (see \eqref{eq:DefNrepMongeSetConv}) polytope for $N$ marginals and $3$ states is visualized for $N=3$ (top-left), $4$ (top-right), $6$ (bottom-left) and $10$ (bottom-right) in green respectively red. The elements $\left(\mu_{ij}\right)_{i,j=1}^3$ of the polytopes are parametrized by their off-diagonal entries $\mu_{12},\mu_{13}$ and $\mu_{23}$. The present figure recovers the illustration of the reduced polytopes in the case of three marginals in \cite{GF18}. If a two-dimensional face of the reduced Monge polytope belongs to the boundary of the reduced Kantorovich polytope the occupied area is depicted in red.}
\label{fig:KanMonge}
\end{figure}
\newline
\newline 
\noindent Taking a look at Figure \ref{fig:KanMonge} the reader may notice that in each of the illustrations one of the extreme points creates a 'peak' in the front of the polytope. As indicated by the coloring each of these extreme points is of Monge-type. In Section \ref{sec:VisComp}, more specifically Theorem \ref{the:ModProp}, these 'peaks' are identified as the unique solutions of OT problems with respect to very rudimentary repulsive cost functions based on the discrete metric. The results of Section \ref{sec:ExtremalCoefficients} originated from the idea 'peaks solve repulsive OT problems'. For any number of marginals $N \geq 2$ and any number of states $\ell \geq 2$ we consider a large class of repulsive pair-costs with the underlying assumption on these cost functions being that their diagonal entries are constant and 'big enough' when compared to their off-diagonal entries. We show that in the given setting any optimizer of problem \eqref{eq:ProblemN} only gives mass to tupels $(x_1,\ldots,x_N)\in X^N$ for which the 'appearance-frequency' of elements of the finite state space $X$ is as uniform as possible (see Theorem \ref{the:SupOptC}). So for $N=5$ and $\ell = 3$ $(a_1, a_2, a_1, a_3, a_2)$ is a valid support-tupel whereas $(a_1,a_2,a_1,a_3,a_1)$ is not. This support-condition is more explicit than the in OT common notion of c-cyclical monotonicity and for certain parameter constellations it paints a very thorough picture of what optimizers look like. In case
\begin{equation}
\label{eq:eqparameterchoice}
N=k\cdot \ell \mathrm{\ for \  some \ } k \in \mathbb{N} \mathrm{\ or \ } \left(N-\left\lfloor \frac{N}{\ell} \right\rfloor \cdot \ell \right) \in \{\ell-1,1\}
\end{equation} 
the discussed condition provides an optimizer in Monge-form and implies its uniqueness. This optimizer represents the depicted 'peak' for $\ell=3$ and a higher-dimensional analogue for $\ell > 3$. Note that for $\ell = 3$ states \eqref{eq:eqparameterchoice} is fulfilled for any $N \geq 2$. So to put it in a nutshell, for any pair of parameters $N,\ell \geq 2$ fulfilling \eqref{eq:eqparameterchoice} we provide a large class of repulsive pair-costs which yield the idea 'peaks solve repulsive OT problems' to be true. As for why this identification of 'peaks' as optimizers is limited to parameter choices $N,\ell$ fulfilling \eqref{eq:eqparameterchoice} note the following. When keeping the number of states $\ell>3$ fixed and letting $N$ tend to $+\infty$ starting at $N=\ell$ the geometric behaviour can be described as follows: While the 'peak' remains intact for $N=\ell$ as well as $N=\ell+1$, for $N=\ell+2$ its representing measure (see Definition \ref{def:NRep}) on the product space $X^N$ blossoms into multiple extreme points; for $N=2\ell-1$ the blossom retracts into a closed state recreating the 'peak'.
\newline
In order to prove Theorem \ref{the:SupOptC}, we establish a lower bound on the nonzero entries of extreme points of the 'coefficient-polytope' $\Pcoef$, which - we believe - is in itself an interesting result.

\section{Classification of the Extreme Points of a Kantorovich Polytope}
\label{sec:ClassKan}
Throughout the paper we will consider the finite state space $X$ given by \eqref{eq:FiniteStateSpace}. We will denote the set of probability measures on $X$ as $\mathcal{P}(X)$. Each such probability measure $\lambda \in \mathcal{P}(X)$ can be canonically identified with a vector in $\mathbb{R}^{\ell}$ via $\lambda_i:=\lambda({a_i})$. The vector $(\lambda_1, \ldots,\lambda_{\ell})$ fulfills $\sum_{i=1}^{\ell} \lambda_i=1$ and $\lambda_i\geq0$ for $i \in \{1,\ldots,\ell\}$ and is therefore an element of the unit simplex. The probability measure $\lambda$ can then be written as $\lambda=\sum_{i=1}^{\ell} \lambda_i\delta_i,$ where here and below we use the shorthand notations
\begin{equation}
\label{eq:ShortDelta}
\delta_i:=\delta_{a_i}, \hspace{2cm} \delta_{i_1\ldots i_N}:=\delta_{a_{i_1}}\otimes \ldots \otimes \delta_{a_{i_N}}
\end{equation}
for $a_i\in X$ being a single point in the finite state space and $(a_{i_1},\ldots,a_{i_N})$ being an element of the product space $X^N$.
\newline
\newline
As announced in the Introduction, we will now take a closer look at the set of admissible trial states of problem 
\eqref{eq:ProblemN}, i.e., the polytope
\begin{equation}
\label{eq:SetKanCou}
\Psymlambda:=\left\{\gamma \in \Psym: \gamma \mapsto \overline{\lambda} \right\}.
\end{equation}
From this point on, we will refer to the elements of this set as \textit{symmetric Kantorovich couplings}. The set $\Psymlambda$ itself will be called \textit{(symmetric) Kantorovich polytope for $N$ marginals and $\ell$ states}. As within this paper we focus our attention on the symmetric case, the term symmetric will be dropped from time to time. It is easy to see that, as a result of the linearity of the marginal constraint and the finiteness of the state space $X$, $\Psymlambda$ is a compact and convex set in $\mathbb{R}^{{\ell}^N}$ and therefore by Minkowski's theorem (see, e.g., \cite{Ho94}) the convex hull of its extreme points. 
\newline
\newline
Recall the following basic definitions and notions of convexity (see, e.g., \cite{Ho94, Ro97}). For $y_1,\ldots,y_n \in \mathbb{R}^m$ and $\lambda_1,\ldots,\lambda_n \geq 0$ such that $\sum_{i=1}^n \lambda_i=1$ 
\begin{equation*}
\lambda_1y_1+\dots+\lambda_ny_n=\sum_{i=1}^n \lambda_iy_i
\end{equation*}
is called a \textit{convex combination of the points} $y_1,\ldots,y_n$. A subset $K\subseteq \mathbb{R}^m$ is called \textit{convex} if for each finite selection of points in $K$ each possible convex combination of these points is again contained in $K$. For a subset $V\subseteq \mathbb{R}^m$ the \textit{convex hull of} $V$, denoted as $\conv(V)$, corresponds to the set of all possible convex combinations of a finite selection of points in $V$. Obviously a set $K\subseteq \mathbb{R}^m$ is convex if and only if it is equal to its convex hull. Finally an element $k$ of the convex set $K\subseteq\mathbb{R}^m$ is called an \textit{extreme point} if $k=\lambda_1y_1+\lambda_2y_2$ for some $y_1,y_2\in K$ and $\lambda_1,\lambda_2>0$ such that $\lambda_1+\lambda_2=1$ implies $y_1=k=y_2$. For a considered convex set $K$ the set of extreme points will from now on be denoted as $\ext(K)$. 
\newline
\newline
As $\Psymlambda$ is equal to the convex hull of its extreme points, we can use the extreme points to describe the convex structure of the set of symmetric Kantorovich couplings. Now it follows by a simple contradiction argument that for any given linear objective function there is always an optimizer that is an extreme point. Moreover, in our setting of finite states spaces, for any extreme point $\gamma^*$ there is a function $c:X^N \to \mathbb{R}$ such that 
\begin{multline}
\label{eq:exposed}
\int_{X^N} c(x_1,\ldots,x_N) d\gamma(x_1,\ldots,x_N) > \int_{X^N} c(x_1,\ldots,x_N) d\gamma^*(x_1,\ldots,x_N) \\\textrm{ for any } \gamma \in \Psymlambda,
\end{multline}
i.e., there is a cost function such that $\gamma^*$ is the unique optimizer of the corresponding OT problem.  This is a result of the fact that $\Psymlambda$ is a bounded polyhedron, i.e., a polytope, of finite dimension and therefore only possesses finitely many extreme points each of whom is itself an exposed point (see, e.g., \cite{Ro97}), i.e., a point in the set $\Psymlambda$ that fulfills \eqref{eq:exposed} for some cost function $c:X^N\to \mathbb{R}$. As for any cost function $c:X^N \to \mathbb{R}$ there is always an optimizer that is an extreme point of $\Psymlambda$ and vice versa for any extreme point $\gamma^*$ of $\Psymlambda$ there is a cost function $c:X^N \to \mathbb{R}$ such that $\gamma^*$ is the unique optimizer, analyzing, how many of the extreme points of $\Psymlambda$ are of Monge-type, is a good approach to investigate the validity of Monge's ansatz. Recall that in the given setting a probability measure $\gamma \in \Psymlambda$ is said to be of Monge-type or in Monge-form if there are $N$ permutations $\tau_1,\ldots,\tau_N:\{1,\ldots,\ell\} \to \{1,\ldots,\ell\}$ such that
\begin{equation*}
\gamma= \sum_{i=1}^{\ell} \frac{1}{\ell} S \delta_{\tau_1(i)\tau_2(i)\ldots\tau_N(i)},
\end{equation*}
where the \textit{symmetrization operator} $S:\mathcal{P}\left(X^N\right) \to \Psym$ is defined by
\begin{equation}
\label{eq:SymOp}
\left(S\gamma \right)\left(A_1\times \dots \times A_N\right)=\frac{1}{N!} \sum_{\sigma\in S_N} \gamma\left(A_{\sigma(1)}\times \dots \times A_{\sigma(N)}\right) \textrm{ for all } A_1,\ldots,A_N \subseteq X
\end{equation}
with $S_N$ being the group of all permutations on the set $\{1,\ldots,N\}$.
\newline
\newline
As in the given setting it is rather inconvenient and long-winded to check whether a given symmetric Kantorovich coupling is of Monge-type or not, we will derive in the following an alternative LP-formulation of problem \eqref{eq:ProblemN}, where Monge-states will correspond exactly to (rescaled) integer points in the corresponding polytope of admissible trial states. We start by taking a closer look at the convex geometry of the set of symmetric probability measures on the product space $X^N$. 
\newline
\newline
Note that a probability measure $\gamma \in \Psym$ is an extreme point of $\Psym$ if and only if it is of the form 
\begin{equation}
\label{eq:ExtrSym}
S\delta_{i_1 \ldots i_N} \textrm{ for some } 1\leq i_1\leq \dots \leq i_N\leq \ell
\end{equation}
(see \cite{FV18}). Therefore symmetric Kantorovich couplings, which are of Monge-type, are an average of $\ell$ not necessarily distinct extremal symmetric probability measures on the product space $X^N$ with respect to the uniform measure. Below we will elaborate further on this characterization of couplings in Monge-form, which will be the basis for identifying Monge-states with the (rescaled) integer points in a certain polytope. 
\newline
\newline
From now on, we will denote the set of extremal symmetric probability measures, i.e., measures of the form \eqref{eq:ExtrSym}, as $\EsymN$. It was shown in \cite{FV18} that $\EsymN$ contains $\tbinom{N+\ell-1}{N}$ elements.
\newline
As for each pair of these extreme points their support is disjoint, one can immediately deduce the following result.
\begin{prop}
$\Psym$ is a simplex, i.e., the extremal symmetric probability measures on $X^N$ are affinely independent. 
\end{prop}
\noindent Hence, for every $\gamma \in \Psym$ there is a unique way to represent $\gamma$ as a convex combination of extremal symmetric probability measures on $X^N$, i.e., there is a unique non-negative coefficient vector $\alpha \in \mathbb{R}^{|\EsymN|}$ fulfilling $\sum\alpha_{i_1 \ldots i_N}=1$ such that
\begin{equation}
\label{eq:ConComRep}
\gamma=\sum_{1\leq i_1 \leq \dots \leq i_N \leq \ell} \alpha_{i_1 \ldots i_N} S\delta_{i_1 \ldots i_N}.
\end{equation}
As the extreme points of $\Psym$ can be parametrized using their one-point marginal,  $\alpha$ can be interpreted as a probability measure on the set of these one-point marginals. 
\newline
Given the $k$-point marginal map $M_k:\mathcal{P}\left(X^N\right) \to \mathcal{P}\left(X^k\right)$ for $1\leq k\leq N-1$ with 
\begin{equation}
\label{eq:kMarMap}
\left(M_k\gamma \right)(A):= \gamma \left(A\times X^{N-k} \right) \textrm{ for all } A\subseteq X^k
\end{equation} 
for $\gamma \in \mathcal{P}\left(X^N\right)$, with the convention $M_N=id$, note that $M_1$ is a bijection from the set of extremal symmetric probability measures on $X^N$, i.e., measures of the form \eqref{eq:ExtrSym}, to the set of $\frac{1}{N}$-quantized probability measures 
\begin{equation}
\label{eq:PQN}
\PQN:=\left\{ \lambda \in \mathcal{P}\left(X\right): \lambda(\{i\}) \in \left\{0,\frac{1}{N},\ldots,1\right\} \right\}
\end{equation}
(see \cite{FV18}). Hereby the one-point marginal of a measure of form \eqref{eq:ExtrSym} is an empirical measure of the indices $(i_1,\ldots,i_N)$, it holds
\begin{equation*}
M_1S\delta_{i_1 \ldots i_N} = \frac{1}{N}\sum_{j=1}^N \delta_{i_j}.
\end{equation*} 
In the following $\psi_N: \PQN \to \EsymN$ will denote the corresponding inverse function.
\newline
This parametrization gives rise to the \textit{coefficients-to-coupling map} $R:\mathcal{P}\left( \PQN \right) \to \Psym$. It maps an arbitrary probability measure $\alpha$ on $\PQN$, which, via the underlying parametrization, corresponds to the coefficients in the representation \eqref{eq:ConComRep}, to the corresponding coupling $\gamma$, i.e., in pedestrian notation
\begin{equation*}
R\alpha=\sum_{\lambda\in \PQN} \alpha_\lambda \psi_N(\lambda),
\end{equation*}
or more elegantly 
\begin{equation*}
R\alpha=\int_{\PQN} \psi_N(\lambda) d\alpha(\lambda).
\end{equation*}
As $\Psym$ is a simplex, $R$ is bijective. This enables us to establish the following isomorphic relationship between two alternative formulations of the set of symmetric Kantorovich couplings. 
\begin{lem}[isomorphic relationship between couplings and coefficients] 
\label{lem:IsoRel}
The coefficients-to-coupling map $R$ maps the polytope 
\begin{equation}
\label{eq:Pcoef}
\Pcoef:=\left\{ \alpha \in \mathbb{R}^{|\EsymN|}: A\alpha = \overline{\lambda}, \alpha \geq 0\right\}
\end{equation}
linearly and bijectively to the set of symmetric Kantorovich couplings, i.e., $\Psymlambda$ defined in \eqref{eq:SetKanCou}. Here $\EsymN$ is the set of extremal symmetric probability measures on $X^N$ and $A$ is the matrix in $\mathbb{R}^{\ell \times |\EsymN|}$, whose columns are given by the elements of $\PQN$, i.e., 
\begin{equation}
\label{eq:DefA}
A:= \begin{pmatrix}
\lambda_1^{(1)}     & \lambda_1^{(2)}      &  & \dots & & \lambda_1^{(|\EsymN|)}\\
\vdots                      & \vdots                     &   &         & & \vdots\\
\lambda_{\ell}^{(1)} & \lambda_{\ell}^{(2)} &  &  \dots & & \lambda_{\ell}^{(|\EsymN|)}
\end{pmatrix}.
\end{equation}
The corresponding inverse map is also linear. 
\end{lem}
\begin{proof}
Linearity and injectivity of $R$ as a map from $\Pcoef$ to $\Psymlambda$ is an immediate consequence of the linearity and injective of $R:\mathcal{P}\left( \PQN \right) \to \Psym$ as introduced above. We further know that any $\gamma \in \Psymlambda$ is an element of $\Psym$. Hence, applying the parametrization of extremal symmetric probability measures on $X^N$ via their one-point marginals,  there exist coefficients $\alpha \in \mathcal{P}\left( \PQN \right)$, which are non-negative and whose entries sum to $1$ such that
\begin{equation}
\label{eq:ProofRepGam}
\gamma = \sum_{\lambda\in \PQN} \alpha_\lambda \psi_N(\lambda)
\end{equation}
and therefore $\gamma = R\alpha$ holds. Applying the linear marginal map $M_1$ to \eqref{eq:ProofRepGam} yields the following. 
\begin{equation*}
\overline{\lambda}= \sum_{\lambda\in \PQN} \alpha_\lambda \lambda
\end{equation*}
Therefore $\alpha$ corresponds to an element of $\Pcoef$. This implies surjectivity of the considered map $R$.
Linearity of the corresponding inverse map is an immediate consequence of the fact that the extremal symmetric probability measures on $X^N$ of the form \eqref{eq:ExtrSym} interpreted as vectors are linearly independent. 
\end{proof}
\noindent Now it is easy to see that the extreme points of $\Pcoef$ correspond exactly to the extremal symmetric Kantorovich couplings, in the sense that $R$ maps the corresponding sets of extreme points bijectively to each other. By standard arguments of polyhedral optimization the extreme points of $\Pcoef$ have a sparse structure, i.e., any extreme point of $\Pcoef$ can have at most $\ell$, that is the number of states in the finite state space $X$, non-zero entries (see, e.g., \cite{Be09}). In \cite{FV18} it was shown that this implies that any extremal Kantorovich coupling is a so called \textit{Quasi-Monge state}, i.e., of the form $\sum_{\nu=1}^{\ell} \alpha_\nu S\delta_{T_1\left(a_{\nu} \right)} \otimes \dots \otimes \delta_{T_N \left( a_{\nu} \right)}$ for $N$ maps $T_1,\ldots,T_N:X\to X$ such that $\frac{1}{N}\sum_{k=1}^N T_k\#\alpha =\overline{\lambda}$. Here we renounce from using the shorthand notations \eqref{eq:ShortDelta} in order to make it easier to draw a comparison with Monge's approach \eqref{eq:MongeOne}-\eqref{eq:MongeThree}. The ansatz space of Quasi-Monge states increases the number of unknowns only by $2\cdot\ell$ compared to the class of symmetrized Monge states and as every extremal Kantorovich coupling is a Quasi-Monge state, this ansatz space always contains an optimal coupling, in contrast to Monge's approach. Note further that obviously every symmetrized Monge state is a Quasi-Monge state. For further reading on this sufficient low-dimensional enlargement of the class of symmetrized Monge states we refer the interested reader to \cite{FV18}. There also a characterization of Monge states in the given setting was established. A probability measure on the product space $X^N$  is a symmetrized Monge state if and only if it is a Quasi-Monge state all of whose site weights $\alpha_1,\ldots,\alpha_\ell$ are equal to $\frac{1}{\ell}$. In summary, we get the following corollary. 
\begin{coro}
\label{cor:NewFor}
Extremal symmetric Kantorovich couplings correspond exactly, via the coefficients-to-coupling map $R$, to the extreme points of $\Pcoef$. Any of these extreme points of $\Pcoef$ is the coefficient vector of a coupling in Monge-form if and only if it is an integer vector scaled by the factor $\frac{1}{\ell}$.
\end{coro}
\noindent This corollary gives us a numerically-convenient way to compute the set of extremal Kantorovich couplings and check whether they are of Monge-form or not. In addition we also want to consider Monge's approach by itself. For this purpose we introduce the sets
\begin{equation}
\label{eq:DefMongeSet}
\PsymMonge:=\left\{ \gamma \in \Psymlambda : \gamma \textrm{ is of Monge-form } \eqref{eq:MongeOne}-\eqref{eq:MongeThree} \right\}
\end{equation}
and 
\begin{equation}
\label{eq:DefMongeSetConv}
\PsymMongeConv:= \conv\left(\PsymMonge\right).
\end{equation}
$\PsymMonge$ is the set of all symmetrized Monge states. In the following we will refer to $\PsymMongeConv$ as the \textit{(symmetric) Monge polytope for $N$ marginals and $\ell$ states}. For simplicity we will once again drop the term symmetric from time to time. Note that if there exists an optimizer of problem \eqref{eq:ProblemN} which is an element of $\PsymMongeConv$ then there exists a Monge-type minimizer. 
\newline
\newline
Having the explanations leading up to Corollary \ref{cor:NewFor} in mind, it is easy to see that $\PsymMonge$ corresponds to the (scaled by $\frac{1}{\ell}$) integer elements of $\Pcoef$. These can be for example determined by a simple enumeration of all the ordered choices of $N-1$ permutations interpreted as coefficient vectors in $\Pcoef$. Checking which of these scaled integer coefficient vectors are extremal with respect to the convex hull of them as a whole, gives us the extremal elements of $\PsymMongeConv$.
\newline
\newline
The data in Figure \ref{fig:SecTwoExt} was computed using MATLAB \cite{MATLAB2018b} and polymake \cite{POLYMAKE32}.
\begin{figure}[htbp]
\centering
  \includegraphics[height=180mm]{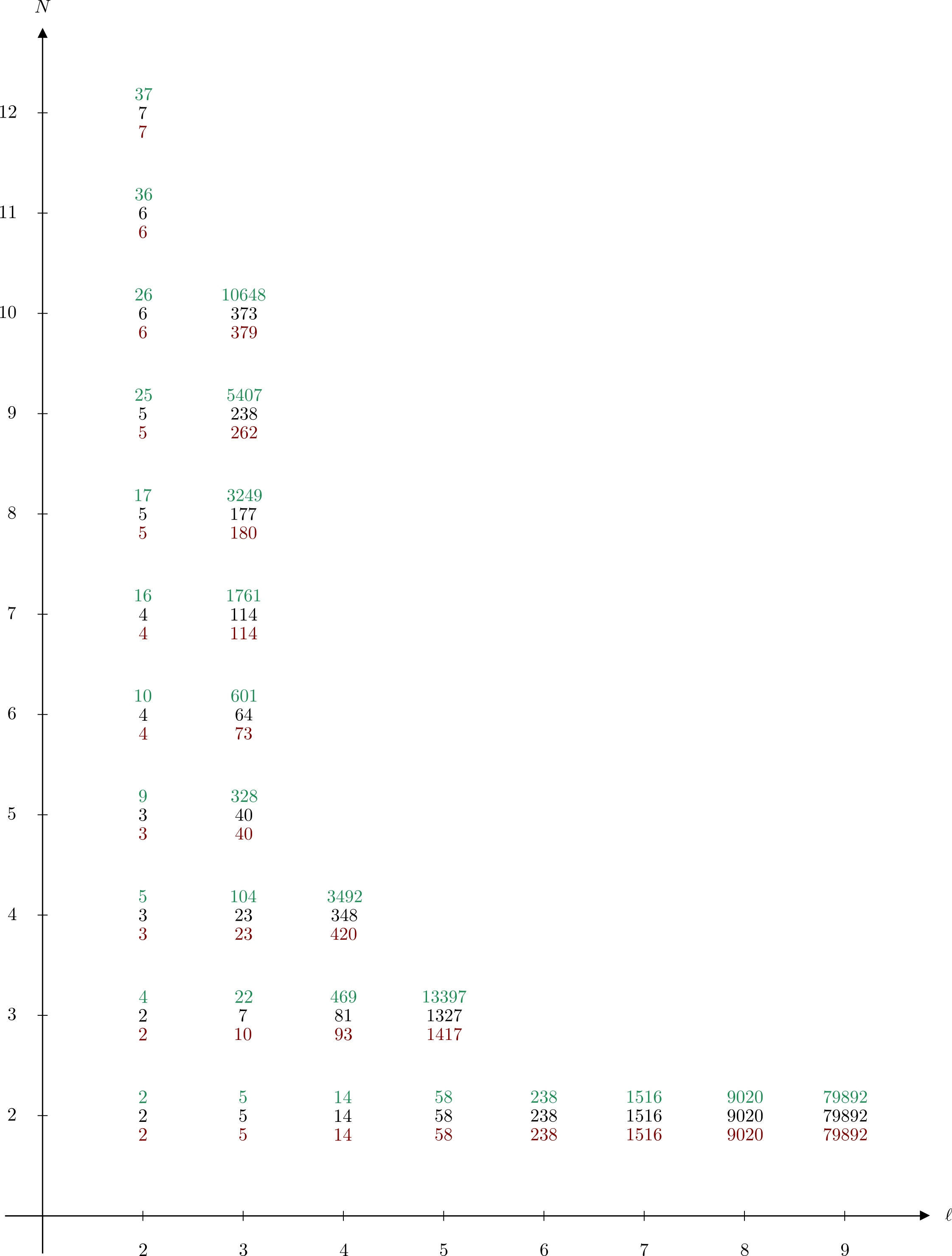}
\caption{The number of extreme points of the symmetric Kantorovich polytope $\Psymlambda$ respectively of the symmetric Monge polytope $\PsymMongeConv$ is given in green respectively red. The number of extreme points of $\Psymlambda$ that are of Monge-type (see \eqref{eq:MongeOne}-\eqref{eq:MongeThree}) is depicted in black. Here, as usual, $N$ denotes the number of marginals and $\ell$ the number of states.}
\label{fig:SecTwoExt}
\end{figure}
\newline
\newline
It was already mentioned above, that the extreme points of the polytope $\Pcoef$ have a sparse structure. In more detail, a coefficient vector $\alpha \in \Pcoef$ is extremal with respect to the polytope $\Pcoef$ if and only if its nonzero entries correspond to a selection of columns of $A$ which are linearly independent (see, e.g., \cite{Be09}). That is why, the complexity of computing the extreme points of $\Pcoef$, and their number,  increases faster with the number of states than with the number of marginals. Suppose you are looking at a setting where the number of marginals is equal to the number of states. Then, on the one hand, increasing the number of marginals by $1$ yields $\tbinom{2N}{N+1}- \tbinom{2N-1}{N}$ more columns in $A$. On the other hand, an increase in the number of states by $1$ enlarges the number of columns of $A$ by $\tbinom{2N}{N}-\tbinom{2N-1}{N}$. Elementary computations show that in the second case $A$ has $\frac{1}{N+1}\tbinom{2N}{N}$ more columns than in the first case. Moreover, in contrast to an increase in the number of marginals, an increase in the number of states also increases the number of rows of $A$ by 1. Therefore then up to $\ell+1$ columns of $A$ can be linearly independent. Hence, an increase in the number of states leads to a faster increasing (compared to an increase in the number of marginals) number of subsets of linearly independent columns of the constraint matrix $A$ by yielding a steeper increase in the number of columns of $A$ as well as by enlarging the dimension of the column space. Each of these subsets corresponds to an extreme point of $\Pcoef$. 
\begin{rem}
\label{rem:IntObsOne}
This remark lists interpretations and observations regarding Figure \ref{fig:SecTwoExt}.
\begin{enumerate}[1)]
\item In the case $N=2$, Figure \ref{fig:SecTwoExt} shows that in the considered cases every extremal symmetric Kantorovich coupling is of Monge-type. In the given setting this means that every extreme point of $\Psymlambda$ is a symmetrized permutation matrix. It is easy to see, using the celebrated Birkhoff-von Neumann theorem \cite{Bi46, vN53} as well as the linearity of the symmetrization operator $S$ \eqref{eq:SymOp}, that this holds true for an arbitrary number $\ell$ of states. Note, however, that not every symmetrized permutation matrix is an extreme point of $\Psymlambda$, but only those symmetrized Monge-states whose corresponding coefficient vectors select linearly independent columns of $A$.  
\item In the case $\ell=2$, the number of extremal symmetric Kantorovich couplings which are of Monge-type increases by 1 each time the marginal number is even. It is easy to prove that this pattern will continue. Firstly, note that, in the case $\ell=2$, every symmetric Kantorovich coupling of Monge-type is an extreme point of $\Psymlambda$. This follows by a support-argument regarding the corresponding coefficient vectors. Secondly, we take a look at the symmetrized Monge-states in this setting. We assume the marginal vectors $\lambda^{(1)},\ldots,\lambda^{(N+1)}$ are sorted in the columns of $A$ by the first component in decreasing order, i.e.,
\begin{equation*}
A=\begin{pmatrix}
1 & \frac{N-1}{N} & \dots & \frac{1}{N} & 0 \vspace{0.2cm}\\ 
0 & \frac{1}{N} & \dots & \frac{N-1}{N} & 1
\end{pmatrix}.
\end{equation*}
Then the symmetric Kantorovich couplings of Monge-type are exactly those couplings with coefficient vectors
\begin{equation*}
\alpha^{(j)}=\frac{1}{2}e^{(j)}+\frac{1}{2}e^{(N+1-j+1)}
\end{equation*}
for $j=1,2,\ldots, \left\lceil \frac{N+1}{2} \right\rceil$, where $e^{(i)}$ is the $i$-th unit vector.
\item
The setting of $3$ marginals and $3$ sites, i.e., $N=\ell=3$ is the main focus in \cite{GF18}. There interested readers can find the $22$ extreme points of the symmetric Kantorovich polytope explicitly listed including the information which extremal elements are of Monge-type and which are not. This list also shows which pairs of permutations (identifiying $T_1$ with the identity) correspond to an extremal symmetric Kantorovich coupling. \cite{GF18} also visualizes these $22$ extremal states as molecular packings, where one can identify irreducible packings with extreme points.  
\item \label{item:Blacktogreen} 
Note that for each grid-point in Figure \ref{fig:SecTwoExt} dividing the number depicted in black by the number depicted in green, i.e., '$\frac{\textrm{black}}{\textrm{green}}$', gives the ratio of extreme points of the symmetric Kantorovich polytope that are of Monge-type. For a fixed three element state space, i.e., $\ell=3$, this ratio consistently decreases with growing $N$ from $1$ for $2$ marginals to $0.035$ for $10$ marginals. Reversing the roles of $N$ and $\ell$, i.e., fixing the number of marginals $N$ to three and letting the number of marginal states grow, also yields a consistently decreasing behavior of the considered ratio; starting from $0.5$ for $\ell=2$ and ending at $0.099$ for $\ell = 5$. \newline
The considered ratio has the following interesting probabilistic interpretation. Given a non-degenerate cost function, i.e., a cost function that yields a unique optimizer of problem \eqref{eq:ProblemN}, the probability of the corresponding optimizer being of Monge-type is given by the considered ratio. Here we obviously draw uniformly from the set of extremal symmetric Kantorovich couplings. Specific cost functions might always yield Monge-type optimizers, see Section \ref{sec:VisComp}. 
\item For each grid-point in Figure \ref{fig:SecTwoExt} the ratio '$\frac{\textrm{black}}{\textrm{red}}$' is an indicator of how much unnecessary information is contained in Monge's ansatz. For the considered cases this ratio is always above $0.8$ except for one outlier at $N=\ell=3$ where the ratio is given by $0.7$. Hence, even though the Monge ansatz does not contain the 'entire information of the Kantorovich polytope', see \ref{item:Blacktogreen}), at least it does not entail 'a lot' of unnecessary information.
\item Finally, we want to give computationally determined examples of non-Monge extreme points in the case of $N=3$ marginals. 
\begin{equation*}
\ell = 3: \ \frac{1}{2}\psi_N \begin{pmatrix}
2/3\\
1/3\\
0
\end{pmatrix} + \frac{1}{2}\psi_N \begin{pmatrix}
0\\
1/3\\
2/3
\end{pmatrix} \hspace{5mm} \ell = 4 : \ \frac{3}{8}\psi_N \begin{pmatrix}
2/3\\
1/3\\
0\\
0
\end{pmatrix} + \frac{3}{8}\psi_N \begin{pmatrix}
0\\
1/3\\
2/3\\
0
\end{pmatrix} + \frac{1}{4}\psi_N \begin{pmatrix}
0\\
0\\
0\\
1
\end{pmatrix}
\end{equation*}
\begin{equation*}
\ell = 5 : \ \frac{3}{10}\psi_N \begin{pmatrix}
2/3\\
1/3\\
0\\
0\\
0
\end{pmatrix} + \frac{3}{10}\psi_N \begin{pmatrix}
0\\
1/3\\
2/3\\
0\\
0
\end{pmatrix} + \frac{1}{5}\psi_N \begin{pmatrix}
0\\
0\\
0\\
2/3\\
1/3\\
\end{pmatrix} + \frac{1}{5}\psi_N \begin{pmatrix}
0\\
0\\
0\\
1/3\\
2/3\\
\end{pmatrix}
\end{equation*}
The $(\ell=3)$-example was already given in \cite{GF18}. The two remaining extreme points both consist of two components. One that is compatible with Monge's approach; it is given by the last $(\ell = 4)$ respectively the two last $(\ell=5)$ terms of the corresponding sum. The remaining component arises from the $(\ell=3)$-example and yields the non-Monge property of the considered extreme points. These considerations indicate how to construct non-Monge extreme points for growing $\ell$: Assume the number of states $\ell$ to be no less than 6. Firstly, choose an increasing triple $\{i_1,i_2,i_3\}$ of pairwise distinct indices from the set $\{1,2,\ldots,\ell\}$. These indices mark the elements of the finite state space $X$ that will be 'covered' by the $(\ell=3)$-example of a non-Monge extreme point from above. In order to simplify notation we assume these 'covered' states to be $a_1,a_2$ and $a_3$, i.e., $\{i_1,i_2,i_3\} = \{1,2,3\}$. Moreover, let $\lambda^{(1)},\ldots,\lambda^{(\ell-3)} \in \mathcal{P}_{\frac{1}{3}}(\{a_4,\ldots,a_\ell\})$ be $\frac{1}{3}$-quantized probability measures on $X\backslash \{a_1,a_2,a_3\}$ which form an extreme point of the symmetric Kantorovich polytope for 3 marginals and $\ell-3$ states in Monge-form. Then 
\begin{equation*}
\frac{3}{2 \cdot \ell}\psi_N \begin{pmatrix}
2/3\\
1/3\\
0\\
0\\
\vdots\\
0
\end{pmatrix} + \frac{3}{2 \cdot \ell}\psi_N \begin{pmatrix}
0\\
1/3\\
2/3\\
0\\
\vdots\\
0
\end{pmatrix} + \sum_{i=1}^{\ell-3}\frac{1}{\ell}\psi_N \begin{pmatrix}
0\\
0\\
0\\
\vline\\
\lambda^{(i)}\\
\vline
\end{pmatrix}
\end{equation*}
is a non-Monge extreme point for 3 marginals and $\ell$ states. The described construction reveals that from the single non-Monge extreme point for $\ell=3$ states arise 
\begin{equation*}
{\ell \choose 3}\cdot \mathrm{Monge}_{\ell-3}
\end{equation*}
non-Monge extreme points for $\ell$ states. Hereby $\mathrm{Monge}_{\ell-3}$ is the number of extreme points of the symmetric Kantorovich polytope for 3 marginals and $\ell-3$ states in Monge-form.
\end{enumerate}
\end{rem}
\begin{figure}[htbp]
\centering
\begin{minipage}[t][85mm][t]{0.48\textwidth}
  \centering
  \includegraphics[width=1.0\linewidth]{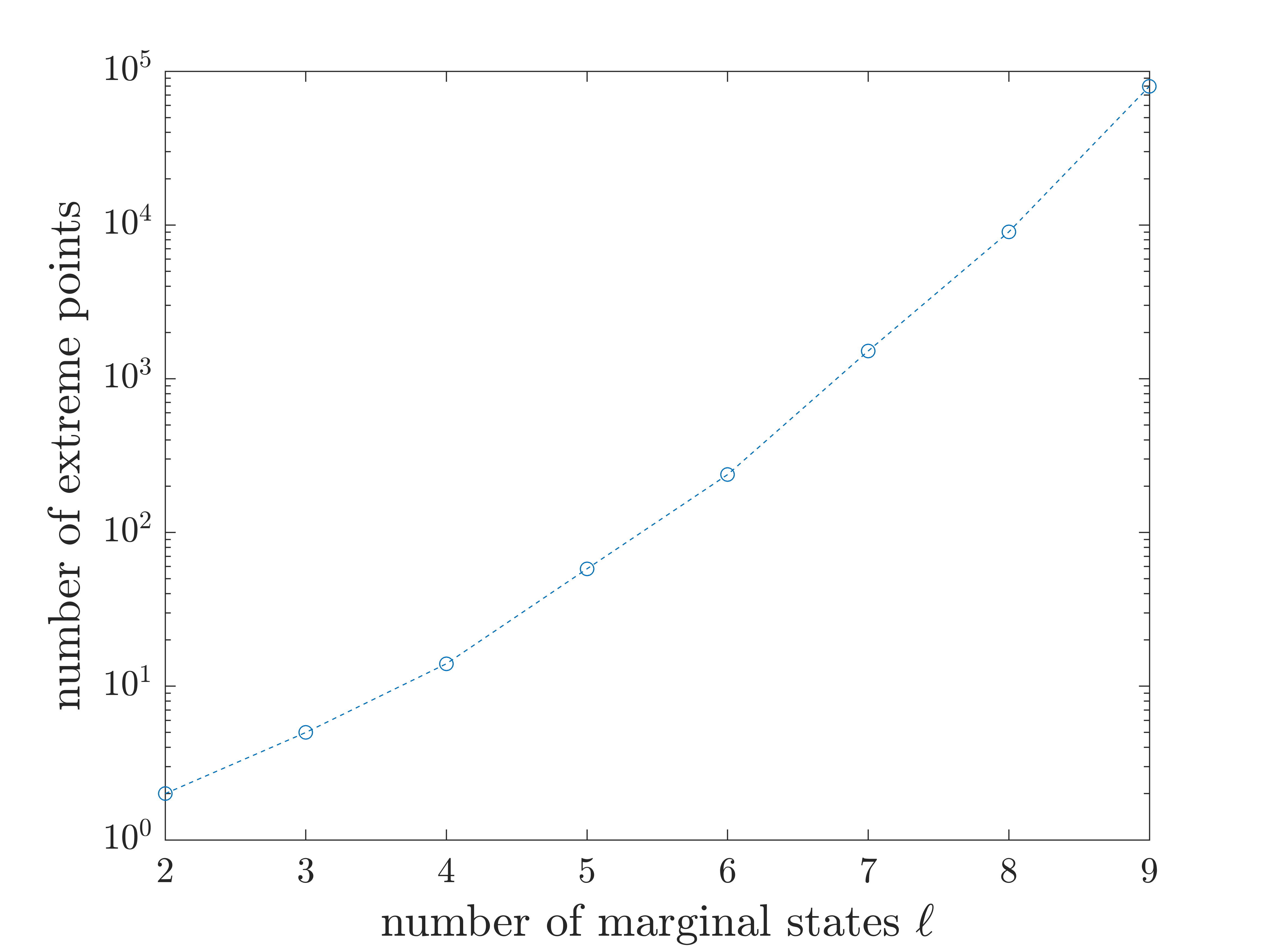}
 \captionof{figure}{The number of extremal symmetric Kantorovich couplings for $N=2$ marginals is depicted in dependency of the number of marginal states $\ell$.}
  \label{fig:Remark1}
\end{minipage}%
\hfill
\begin{minipage}[t][85mm][t]{.48\textwidth}
  \centering
  \includegraphics[width=1.0\linewidth]{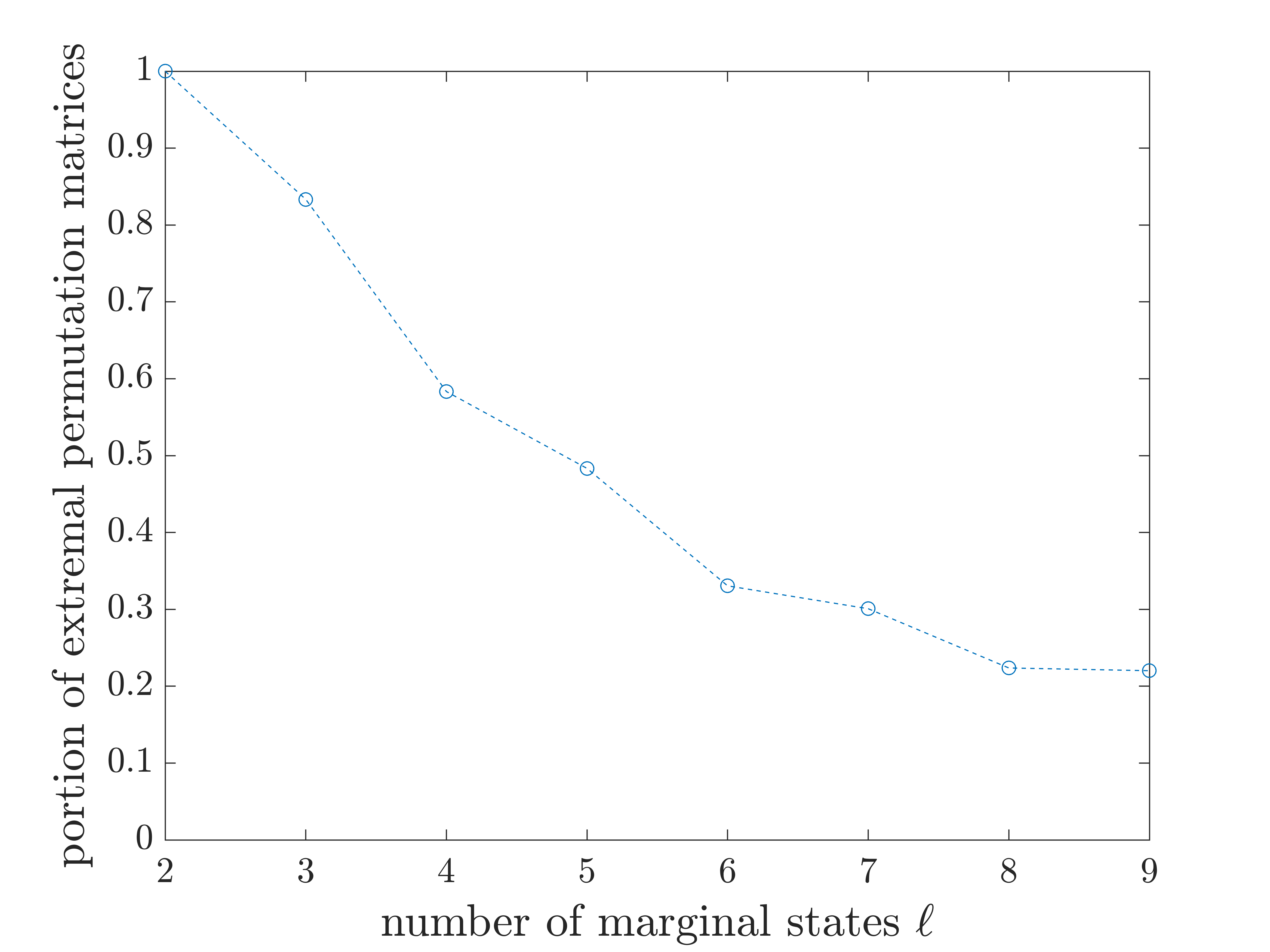}
  \captionof{figure}{The ratio of the number of extremal symmetrized permutation matrices to the number of permutation matrices is depicted in dependency of the number of marginal states $\ell$.}
 \label{fig:Remark2}
\end{minipage}
\end{figure} 
\noindent The computational restriction regarding our extreme point investigation is visualized in Figure \ref{fig:Remark1} and \ref{fig:Remark2}: Figure \ref{fig:Remark1} depicts a super-exponential growth of the number of extremal symmetric Kantorovich couplings in the case of $N=2$ marginals. Moreover, Figure \ref{fig:Remark2} indicates that the portion of permutation matrices which correspond to extremal symmetric Kantorovich couplings (in the case of $N=2$ marginals) tends to a value close to $0.2$.
\section{Classification of the Extreme Points of a Reduced Kantorovich Polytope}
\label{sec:IntroRed}
In Section \ref{sec:ClassKan}, we achieved a better understanding of the OT problem \eqref{eq:ProblemN} by numerically analyzing the convex geometry of the set of admissible trial states, i.e., the set of symmetric Kantorovich couplings $\Psymlambda$. Motivated by physical applications we assume from this point on that the given cost function has pairwise symmetric structure. Then the set of admissible trial states can be reduced by the linear map $M_2$ to a lower-dimensional polytope thereby decreasing the number of extremal states. 
\newline
\newline
In more detail, we consider the OT problem \eqref{eq:ProblemN} with $c:X^N\to \mathbb{R}$ being a cost function with pairwise symmetric structure, i.e.,
\begin{equation}
\label{eq:DefPairSymCost}
c(x_1,\ldots,x_N)=\sum_{1\leq i < j \leq N} v(x_i,x_j)  \textrm{ for all } (x_1,\ldots, x_N) \in X^N, 
\end{equation}
where $v:X^2 \to \mathbb{R}$ is a symmetric pair-potential, i.e., $v(x,y)=v(y,x)$ for all $(x,y) \in X^2$. Then the objective function of \eqref{eq:ProblemN} can be rewritten as 
\begin{equation}
\label{eq:IntRef}
\int_{X^N} c(x_1,\ldots,x_N) d\gamma(x_1,\ldots,x_N)= \binom{N}{2} \int_{X^2} v(x,y) d\left(M_2\gamma\right)(x,y),
\end{equation}
where $\gamma \in \Psymlambda$ is an arbitrary symmetric Kantorovich coupling. This elementary reformulation was established in \cite{FMPCK13}. There also the concept of $N$-representability (see Definition \ref{def:NRep}) was introduced which we will use in the following to write the reduced set of admissible trial states in a compact manner. 
\begin{defi}[$N$-representability]
\label{def:NRep}
A probability measure $\mu \in \mathcal{P}\left(X^k\right)$ is called $N$-representable if there exists a symmetric probability measure $\gamma$ on the product space $X^N$, i.e., $\gamma \in \Psym$, such that $\mu$ is its $k$-point marginal, i.e., 
\begin{equation}
\label{eq:DefNRep}
\mu = M_k\gamma.
\end{equation}
Any such symmetric probability measure on $X^N$ that fulfills \eqref{eq:DefNRep} is then called a representing measure of $\mu$. In the following the set of $N$-representable $k$-plans will be denoted by $\PNrepk$.
\end{defi}
\noindent As we consider pairwise interactions, we will focus our attention on the set of $N$-representable $2$-point measures, i.e., $\PNreptwo$. Note, however, that cost functions $c$ embodying $k$-particle interactions would give rise to a problem reformulation reducing the set of admissible trial states to a subset of $\PNrepk$. In the case $k=N$, $c$ would be a symmetric cost which are, as mentioned in the introduction, 'dual' to the set of symmetric probability measures on the product space $X^N$. In the same manner, cost functions with symmetric pairwise structure have a dual relationship with the set of $N$-representable $2$-plans.
\newline
\newline
By definition, the set of $N$-representable $2$-point measures is the image of the set of symmetric probability measures on $X^N$ under the map $M_2$, defined in \eqref{eq:kMarMap}, i.e., $M_2\left( \Psym \right) =\PNreptwo$. Combining this equality with \eqref{eq:IntRef} yields that \eqref{eq:ProblemTwo} is an equivalent reformulation of the multi-marginal OT problem \eqref{eq:ProblemN} for a cost function with pairwise symmetric structure \eqref{eq:DefPairSymCost}. Here \eqref{eq:ProblemTwo} can also be written as
\begin{equation*}
\min_{\mu \in \PNreplambda} \int_{X^2} v(x,y)  d\mu(x,y),
\end{equation*}
where $\PNreplambda$ is the set of $N$-representable $2$-plans having uniform marginal, i.e., 
\begin{equation}
\label{eq:DefPNreplambda}
\PNreplambda:=\left\{ \mu \in \PNreptwo: M_1(\mu) = \overline{\lambda} \right\}.
\end{equation}
We will refer to the set $\PNreplambda$ as \textit{reduced Kantorovich polytope for $N$ marginals and $\ell$ states}. The convex geometry of this set will be numerically analyzed in the following. Thereby the validity of Monge's approach in the given setting will be tested. 
\newline
\newline
\noindent We have seen above that under the assumption of pairwise symmetric cost functions the OT problem \eqref{eq:ProblemN}, where the set of admissible trial states is given by the high-dimensional set $\Psymlambda$, can be equivalently formulated as a minimization problem over the lower-dimensional set $\PNreplambda$ (see \eqref{eq:ProblemTwo}). The pairwise symmetric structure implies that any symmetric Kantorovich coupling influences the value of the objective function of problem \eqref{eq:ProblemN} only through their respective two-point marginal (see \eqref{eq:IntRef}). The nature of this reformulation, applying the two-point marginal map $M_2$ on the set of symmetric Kantorovich couplings, however, entails that the new set of admissible trial states, i.e., the reduced Kantorovich polytope is only implicitly known. Only in the two-marginal (N=2) case the reduced Kantorovich polytope can be understood in a straightforward manner: It corresponds to the set of symmetric bistochastic matrices scaled by the factor $\frac{1}{\ell}$ (see Remark \ref{rem:IntObsTwo} \ref{rem:IntObsTwoOne}) below for further consideration of the two-marginal case). Hence, in the case $N=2$, $\PNreplambda=\Psymlambda$ holds. For a better understanding of the multi-marginal ($N>2$) case, we will in the following, as motivated, view the reduced Kantorovich polytope as the image of the set of symmetric Kantorovich couplings on $X^N$ under the two-point marginal map, i.e.,
\begin{equation}
\label{eq:ImTwoMar}
M_2(\Psymlambda)=\PNreplambda.
\end{equation} 
\newline
\noindent As described in Section \ref{sec:ClassKan}, $\Psymlambda$ corresponds to the convex hull of its extreme points. Combining this fact with \eqref{eq:ImTwoMar} and the linearity of $M_2$ yields that the reduced Kantorovich polytope $\PNreplambda$ is equal to the convex hull of the two-point marginals of extremal symmetric Kantorovich couplings, i.e., 
\begin{equation}
\label{eq:PNreplambdaConvHull}
\PNreplambda= \conv\left( \left\{ M_2\gamma: \gamma \textrm{ is an extreme point of } \Psymlambda \right\} \right).
\end{equation}
The following proposition is an immediate consequence. 
\begin{prop}
\label{prop:ExtRep}
Any extreme point of the reduced Kantorovich polytope for $N$ marginals and $\ell$ states is the two-point marginal of an extremal symmetric Kantorovich coupling. 
\end{prop}
\noindent Now the question is whether or not $M_2$ represents a bijective relationship between the sets of extreme points of $\Psymlambda$ and $\PNreplambda$. The following remark sheds light on this issue applying the bijective relationship between $\Psymlambda$ and the polytope $\Pcoef$ established in Lemma \ref{lem:IsoRel} and Corollary \ref{cor:NewFor}. 
\begin{rem}
\label{rem:ConPcoefPNrep}
In Section \ref{sec:ClassKan} the extreme points of $\Psymlambda$, i.e., the set of admissible trial states of problem \eqref{eq:ProblemN}, are determined using the set's bijective relationship, captured in the coefficients-to-coupling map $R$ introduced in Section \ref{sec:ClassKan}, to the polytope $\Pcoef$. As explained above in more detail, the map $R$ identifies any symmetric probability measure on $X^N$ $\gamma$ with a coefficient vector $\alpha$, such that $\gamma$ can be written as the respective convex combination of the extreme points of $\Psym$, i.e., \eqref{eq:ConComRep} holds. These coefficients are unique due to the disjoint support of the extremal symmetric probability measures on $X^N$.  It was proven in \cite{FV18} that the two-point marginal map $M_2$ is a bijection between the sets of extreme points of $\Psym$ and $\PNreptwo$, respectively. Due to the linearity of $M_2$, given a coefficient vector $\alpha$ and a symmetric probability measure $\gamma$ on $X^N$, such that $\gamma=R\alpha$, i.e., \eqref{eq:ConComRep} holds true, then
\begin{equation*}
M_2\gamma =\sum_{1\leq i_1\leq \dots \leq i_N \leq \ell} \alpha_{i_1,\ldots,i_N} M_2S\delta_{i_1,\ldots,i_N}.
\end{equation*}
Only now, these coefficients $\alpha$ representing $M_2\gamma$ as a convex combination of the extreme points of the set of $N$-representable two-point measures may not be unique, rendering us unable to identify the extreme points of the reduced Kantorovich polytope with those of the coefficient-polytope $\Pcoef$.
\end{rem}

\noindent The remark above illuminates why the extreme points of the set of symmetric Kantorovich couplings $\Psymlambda$ can not be identified with the extremal elements of the reduced Kantorovich polytope $\PNreplambda$ via $M_2$. The two-point marginal map may for example map multiple extreme points of the set $\Psymlambda$ on a single point of $\PNreplambda$; this point may lie on a face or in the interior of $\PNreplambda$ (see \cite{GF18} for an well-illustrated example). 
\newline
\newline
Nevertheless, it was established in Proposition \ref{prop:ExtRep} that every extremal element of $\PNreplambda$ has a representing measure that is itself extremal with respect to $\Psymlambda$. The extreme points of this set of symmetric Kantorovich couplings were in Corollary \ref{cor:NewFor} identified with the extreme points of $\Pcoef$. In combination with the in Remark \ref{rem:ConPcoefPNrep} established connection between $\Pcoef$ and $\PNreplambda$ this leads us to the following approach to determine the extremal elements of $\PNreplambda$: 
\begin{enumerate}
\item We start by determining the extremal elements of $\Pcoef$. This was already done within the considerations of Section \ref{sec:ClassKan}. 
\item Every such extreme point is multiplied by the matrix $T\in \mathbb{R}^{\ell^2\times |\EsymN|} $ which is constructed as follows. The matrix $A$ as defined in \eqref{eq:DefA} lists all the elements of $\PQN$ as columns. It was proven in \cite{FV18} that for any element $\lambda$ of $\PQN$ the following holds: 
\begin{equation}
\label{eq:ForTwoMar}
M_2\psi_N\left(\lambda\right) = \frac{N}{N-1}\lambda\otimes\lambda - \frac{1}{N-1}\left(\textrm{id,id}\right) \# \lambda, 
\end{equation}
where the map $\psi_N$ was introduced in Section \ref{sec:ClassKan}. Note that it was further established in \cite{FV18} that measures of form \eqref{eq:ForTwoMar} for $\lambda \in \PQN$ are exactly the extreme points of $\PNreptwo$. Now we construct $T$ by replacing any column $\lambda$ of $A$ with $M_2\psi_N\left(\lambda\right)$ as given in \eqref{eq:ForTwoMar} where we canonically identify matrices with vectors by gluing columns together. 
\item Finally we check which points of the form 
\begin{equation*}
T\alpha \textrm{ for } \alpha \in  \ext\left(\Pcoef\right)
\end{equation*} 
are extremal with respect to $\conv\left(\left\{ T\alpha: \alpha \in \ext\left(\Pcoef\right) \right\} \right)$ and therefore by \eqref{eq:PNreplambdaConvHull} with respect to $\PNreplambda$. 
\end{enumerate}
Note that it is computationally more complex to determine the extremal elements of $\PNreplambda$ than those of $\Psymlambda$.
\newline
\newline
Now, we will incorporate Monge's approach in the reduced setting. 
\begin{defi}
\label{def:RedMonge}
An element of the reduced Kantorovich polytope for $N$ marginals and $\ell$ states is said so be \textit{of Monge-type} or \textit{in Monge-form} if it has a representing measure that is of Monge-form (see \eqref{eq:MongeOne}-\eqref{eq:MongeThree}). 
\end{defi}
\noindent This definition is consistent with our goal to check the validity of Monge's approach as any optimizer in Monge-form for problem \eqref{eq:ProblemTwo}  guarantees the existence of an optimizer in Monge-form for problem \eqref{eq:ProblemN}. The set of all elements of $\PNreplambda$ which are in Monge-form will be denoted as $\PNrepMonge$, i.e.,
\begin{equation*}
\PNrepMonge:=\left\{M_2\gamma: \gamma \textrm{ is of Monge-type } \eqref{eq:MongeOne}-\eqref{eq:MongeThree} \right\}.
\end{equation*}
Analogously to \eqref{eq:DefMongeSetConv} we introduce the \textit{reduced Monge polytope for $N$ marginals and $\ell$ states} $\PNrepMongeConv$ as follows. 
\begin{equation}
\label{eq:DefNrepMongeSetConv}
\PNrepMongeConv:=\conv\left( \PNrepMonge \right)
\end{equation}
\newline
The extremal elements of the reduced Monge polytope can be determined in the same manner as those of the reduced Kantorovich polytope (see the description of the procedure above). Starting point are now the extremal elements of the Monge polytope $\PsymMonge$ interpreted as coefficient vectors. 
\newline
\newline
Checking which of the extreme points of the reduced Kantorovich polytope  $\PNreplambda$ correspond to an extremal element of the reduced Monge polytope $\PNrepMongeConv$ tells us which of the extreme points of $\PNreplambda$ are of Monge-type. 
\newline
\newline
The data in Figure \ref{fig:ExtRed} was computed using MATLAB \cite{MATLAB2018b} and polymake \cite{POLYMAKE32}. 
\begin{figure}[htbp]
\centering
  \includegraphics[height=180mm]{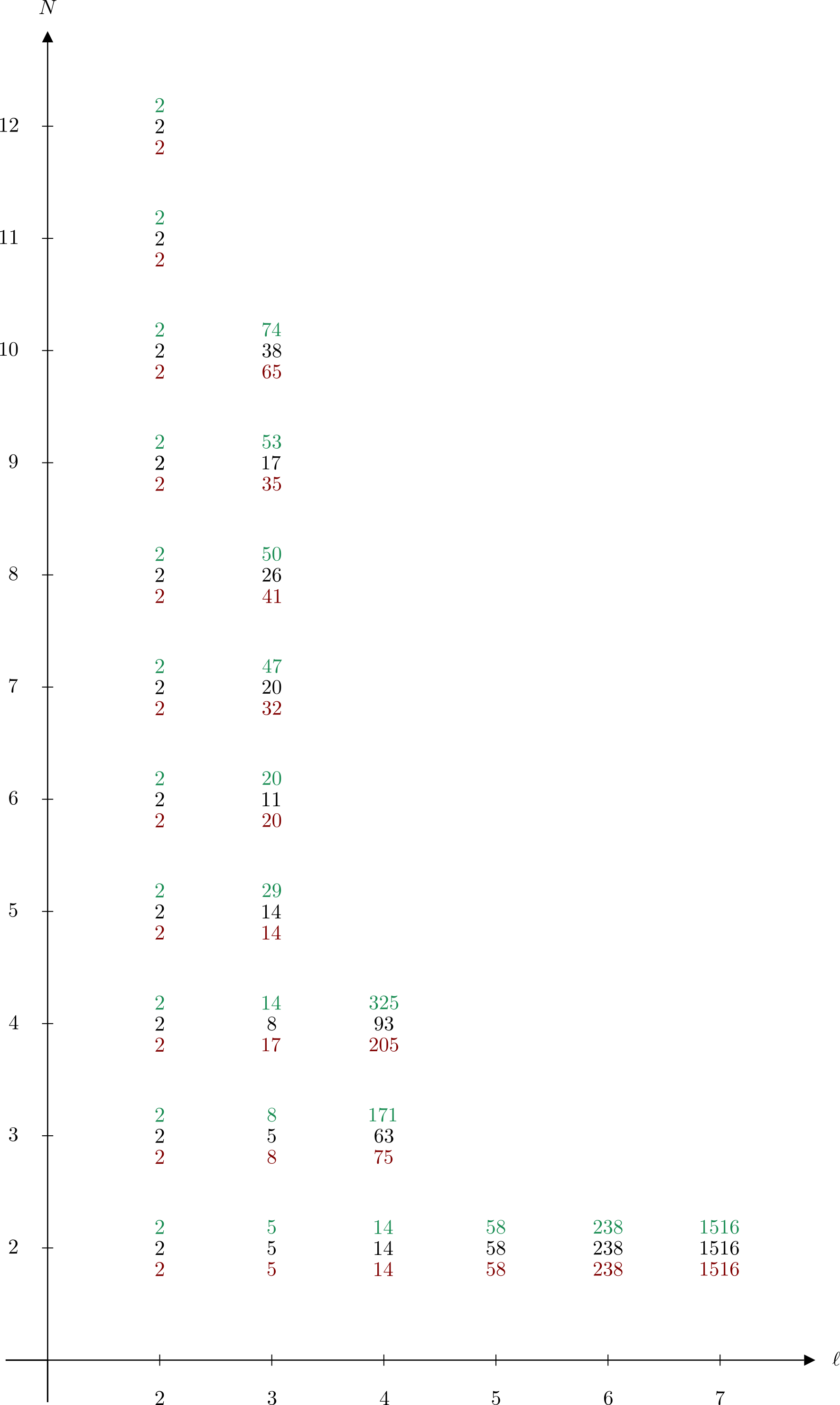}
\caption{The number of extreme points of the reduced Kantorovich polytope $\PNreplambda$ respectively of the reduced Monge polytope $\PNrepMongeConv$ is given in green respectively red. The number of extreme points of $\PNreplambda$ that are of Monge-type (see Definition \ref{def:RedMonge}) is depicted in black. Here, as usual, $N$ denotes the number of marginals and $\ell$ the number of states.}
\label{fig:ExtRed}
\end{figure}
\begin{rem}
\label{rem:IntObsTwo}
What follows are interpretations and observations regarding Figure \ref{fig:ExtRed}.
\begin{enumerate}[1)]
\item \label{rem:IntObsTwoOne}
Combining the convention $M_N= \textrm{id}$ with \eqref{eq:ImTwoMar}, it is obvious that the symmetric Kantorovich polytope for $2$ marginals and $\ell$ states $\mathcal{P}_{\textrm{sym,}\overline{\lambda}}\left( X^2 \right)$ coincides with the reduced Kantorovich polytope for $2$ marginals and $\ell$ states $\mathcal{P}_{2\textrm{-rep,}\overline{\lambda}}\left( X^2 \right)$. 
This fact was already mentioned above.  It was established in Remark \ref{rem:IntObsOne} that every extreme point of $\mathcal{P}_{\textrm{sym,}\overline{\lambda}}\left( X^2 \right)$ is a symmetrized permutation matrix, i.e., the image of a permutation matrix under the symmetrization operator \eqref{eq:SymOp}. In the setting of $2$ marginals, symmetrized permutation matrices exactly correspond to symmetrized Monge states. See Remark \ref{rem:IntObsOne} for further considerations of the case $N=2$. 
\item \label{rem:IntObsTwoTwo}
In the case $\ell=2$, Figure \ref{fig:ExtRed} depicts that in the considered cases, the reduced Kantorovich polytope $\PNreplambda$ has two extreme points both of which are in Monge-form. Hence, in any considered case the line segment $\PNreplambda$ coincides with the respective reduced Monge polytope $\PNrepMongeConv$. 
\newline 
One can prove by elementary arguments that this holds true for an arbitrary number of marginals $N$ in the case of $2$ sites. In a little more detail, considering the dimension of $\PNreplambda$ in the given case and parametrising the elements of $\PNreplambda$ by their off-diagonal element allows us to deduce that the two extreme points of $\PNreplambda$  are given by 
\begin{equation}
\label{eq:TwoStatesExtrOne}
\mu^{(1)} = M_2\left(\frac{1}{2} \psi_N\left( \delta_1\right)+\frac{1}{2}\psi_N\left(\delta_2\right)\right) \end{equation}
\begin{equation}
\label{eq:TwoStatesExtrTwo}
\mu^{(2)}= \begin{cases} M_2\left( \psi_N \left( \frac{1}{2}\delta_1+\frac{1}{2} \delta_2\right) \right) & \textrm{ if } N \textrm{ is even}\\
M_2\left( \frac{1}{2} \psi_N \left( \frac{N-1}{2N}\delta_1+\frac{N+1}{2N} \delta_2\right) +\frac{1}{2} \psi_N \left( \frac{N+1}{2N}\delta_1+\frac{N-1}{2N} \delta_2\right) \right) & \textrm{ if } N \textrm{ is odd}
\end{cases}
\end{equation}
or in pedestrian notation,
\begin{equation*}
\mu^{(1)} = \begin{pmatrix} \frac{1}{2} & 0 \\ 0 & \frac{1}{2}\end{pmatrix}
\end{equation*}
\begin{equation*}
\mu^{(2)}= \begin{cases} \frac{1}{4(N-1)}\begin{pmatrix} N-2 & N \\ N & N-2\end{pmatrix} & \textrm{ if } N \textrm{ is even}\\
\frac{1}{4N}\begin{pmatrix} N-1 & N+1 \\ N+1 &N-1\end{pmatrix} & \textrm{ if } N \textrm{ is odd}. 
\end{cases}
\end{equation*}
As the coefficients in \eqref{eq:TwoStatesExtrOne} and \eqref{eq:TwoStatesExtrTwo} are integer multiples of $\frac{1}{\ell}$ both extreme points are $2$-point marginals of symmetric Kantorovich couplings in Monge-form and therefore they are themselves elements of the reduced Kantorovich polytope $\PNreplambda$ which are of Monge-type (see Definition \ref{def:RedMonge}). Note that $\mu^{(2)}$, which is of Monge-type and therefore describes a correlated or in other words deterministic state, converges for $N \to \infty$ to the independent measure $\lambda \otimes \lambda$ for $\lambda=\left( \frac{1}{2}\delta_1+\frac{1}{2}\delta_2\right)$. 
\newline
These findings coincide with the results in \cite{FMPCK13}, where a model problem of $N$ particles on $2$ sites was considered. There also the set of $N$-representable $2$-plans $\PNreptwo$ for $X$ consisting of $2$ distinct elements is illustrated. Imposing the here given marginal condition on these sets leads to the respective line segment $\PNreplambda$.  
\item
The case of $3$ marginals and $3$ sites, i.e., $N=\ell=3$, is a minimal example of a point in the grid, both with respect to the sum of both parameters $N+\ell$ and with respect to the minimum of both parameters $\min\{N,\ell\}$, such that not every extremal element of the reduced Kantorovich polytope $\PNreplambda$ is of Monge-type. In the considered case $\PNreplambda$ has $8$ extreme points $5$ of which are in Monge-form. By extension $3$ of them are not. They are given by 
\begin{equation}
\label{eq:MinCount}
\frac{1}{2} M_2\left(S\delta_{112}\right) + \frac{1}{2}M_2\left(S\delta_{233}\right)
\end{equation}
and the two states one generates by imposing the role of the second site on the first and third site respectively. \eqref{eq:MinCount} is the unique optimizer of an OT problem stated in \cite{GF18}. This problem corresponds to a molecular packing problem. See \cite{GF18} for further reading. 
\end{enumerate}
\end{rem}
\section{A Model Problem: Optimal Couplings of $N$ Marginals on $3$ Sites for Pairwise Costs}
\label{sec:VisComp}
In the following, we focus our attention on symmetric multi-marginal OT problems \eqref{eq:ProblemN} on $3$ sites, i.e., $X=\{a_1,a_2,a_3\}$. As in Section \ref{sec:IntroRed}, we only consider pairwise symmetric costs and therefore are able to reformulate \eqref{eq:ProblemN} as the lower-dimensional problem \eqref{eq:ProblemTwo}. In particular, the reduced Kantorovich polytope for $N$ marginals and $3$ states corresponds to the respective set of admissible trial states. It is easy to see that in the given setting these polytopes are three-dimensional. As by extension the reduced Monge polytope is at most three-dimensional, we are able to visually compare both approaches. 
\newline
\newline
The following visualizations in Figure \ref{fig:VisComp} were generated by extending the above explained calculations and routines in MATLAB \cite{MATLAB2018b} and polymake \cite{POLYMAKE32}.
\begin{figure}[htbp]
\captionsetup[subfigure]{labelformat=empty}
\centering
\hspace{1mm}
\subfloat{\includegraphics[width=0.06\linewidth]{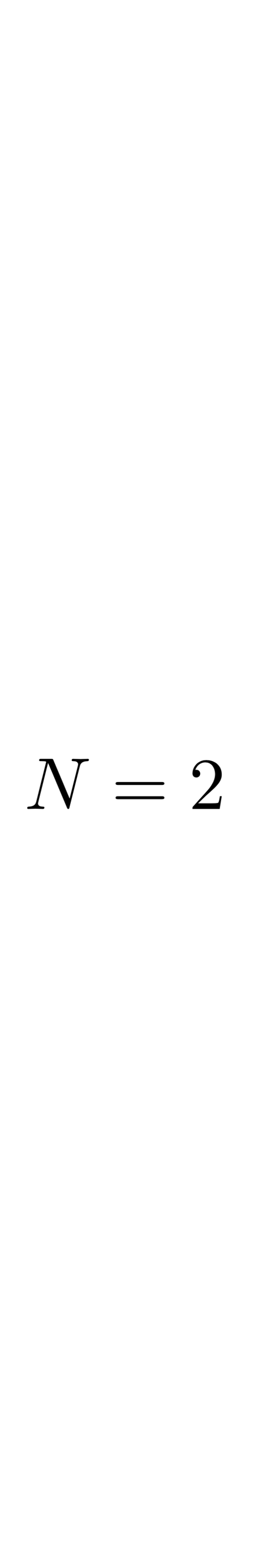}}
\hspace{1mm}
\subfloat{\includegraphics[width=0.44\linewidth]{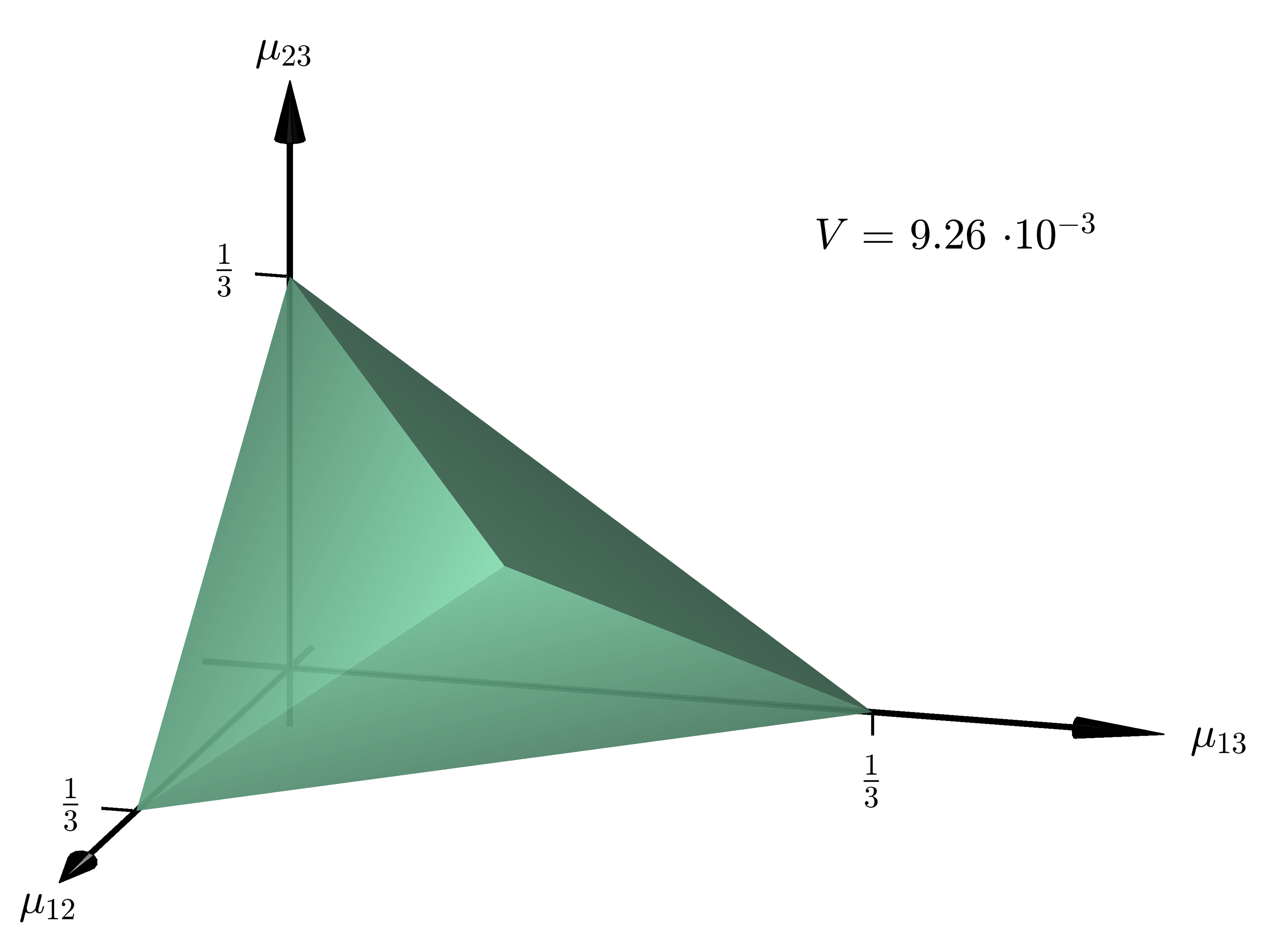}}
\hspace{1mm}
\subfloat{\includegraphics[width=0.44\linewidth]{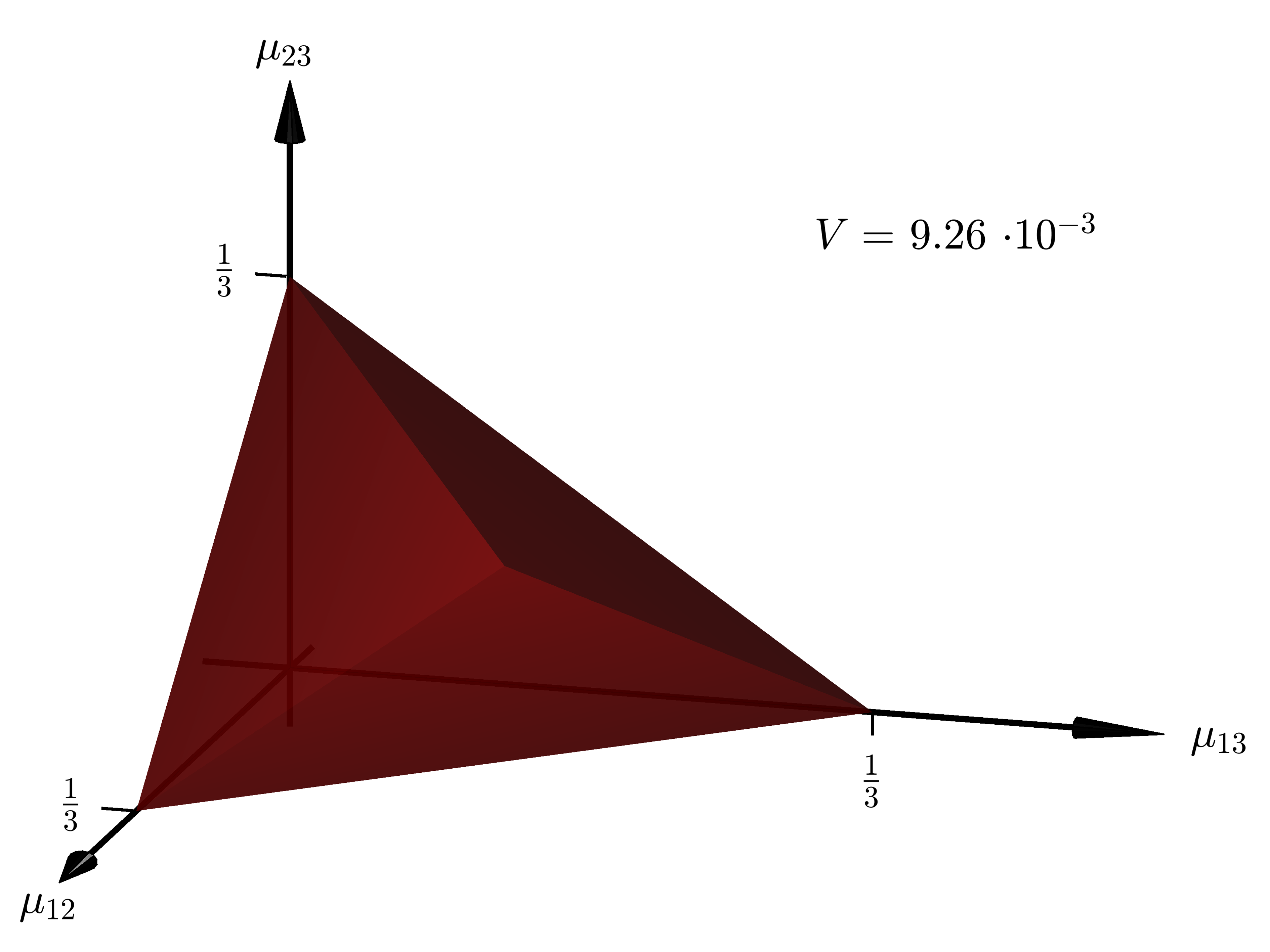}}
\hfill
\\
\hspace{1mm}
\subfloat{\includegraphics[width=0.06\linewidth]{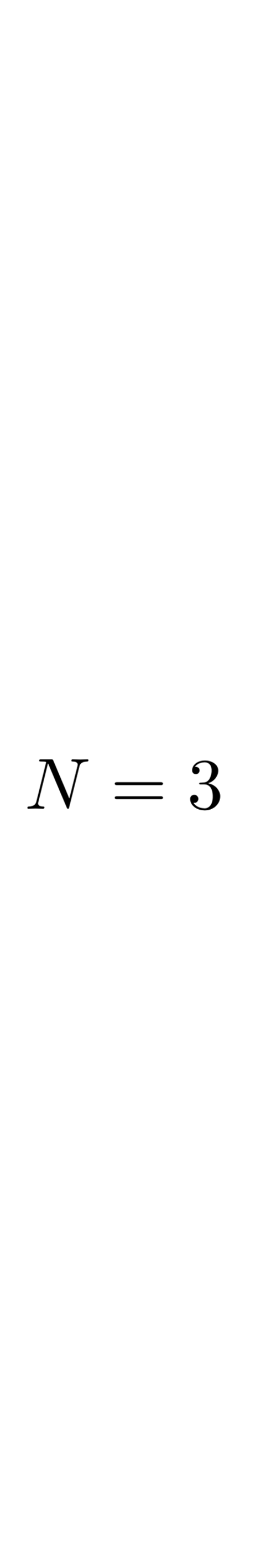}}
\hspace{1mm}
\subfloat{\includegraphics[width=0.44\linewidth]{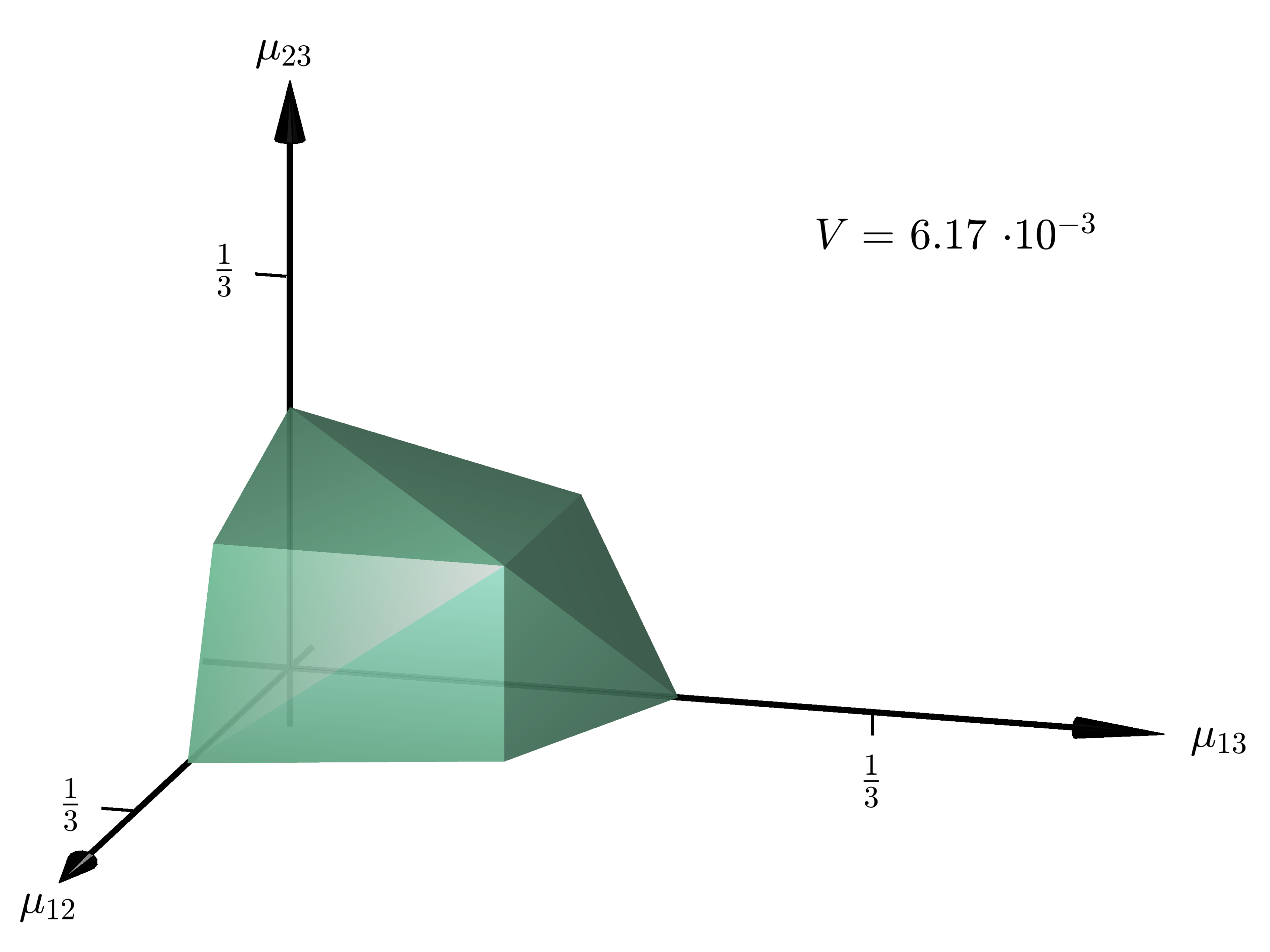}}
\hspace{1mm}
\subfloat{\includegraphics[width=0.44\linewidth]{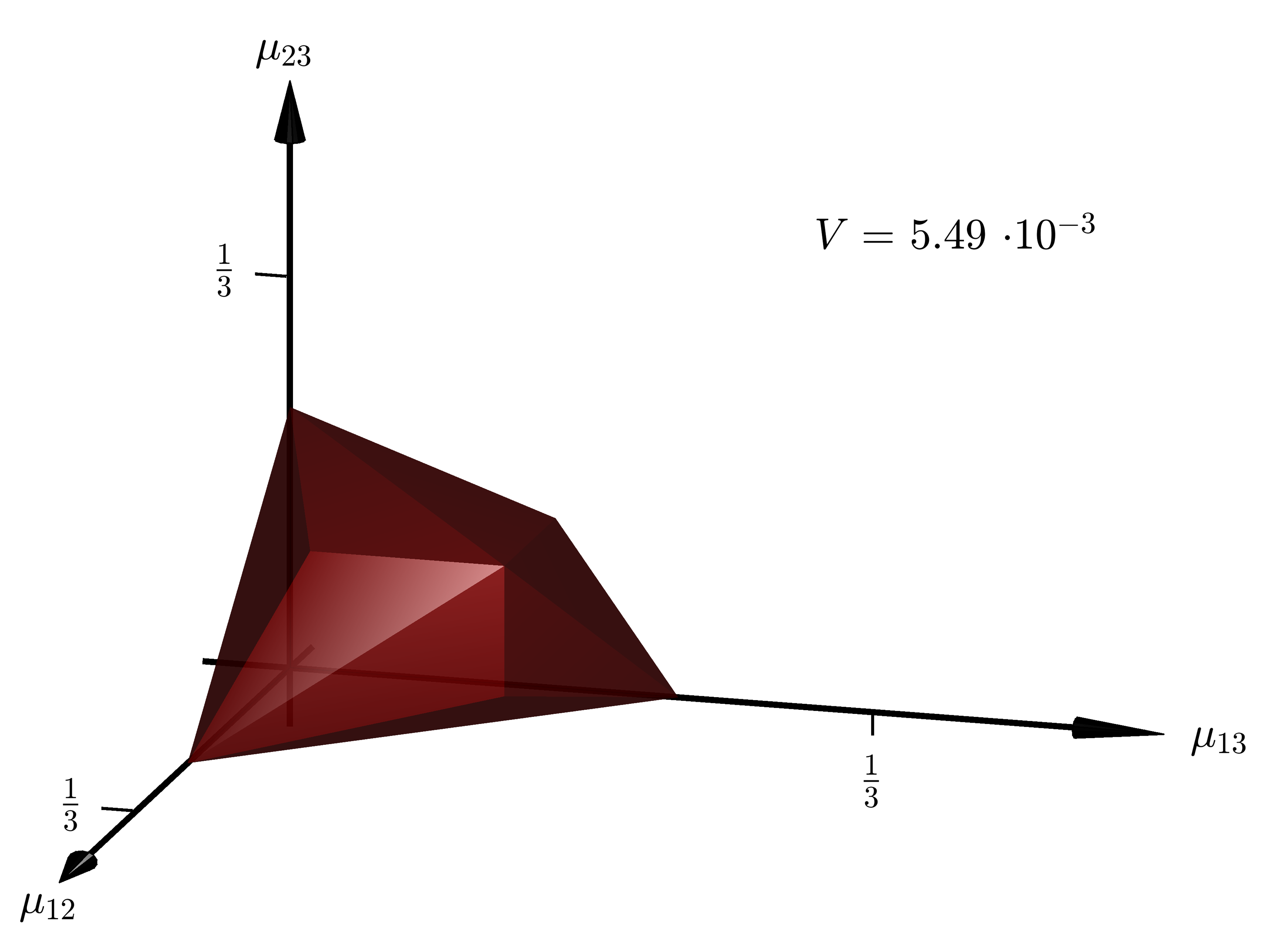}}
\hfill
\\
\hspace{1mm}
\subfloat{\includegraphics[width=0.06\linewidth]{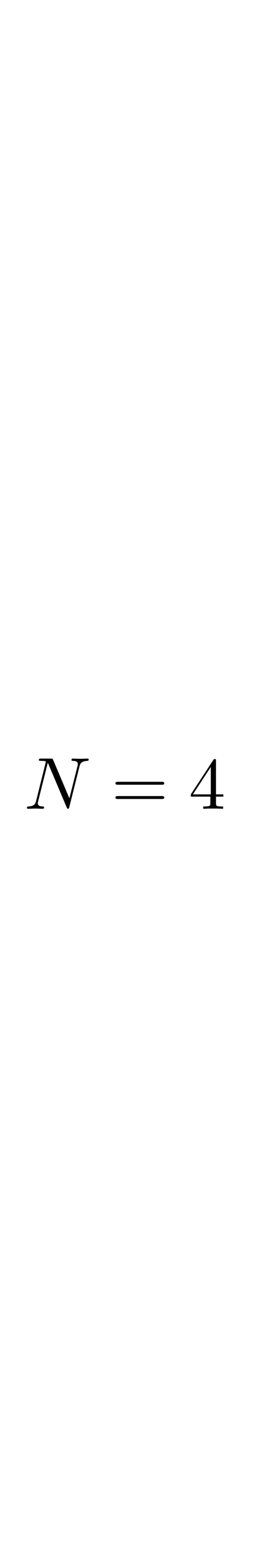}}
\hspace{1mm}
\subfloat{\includegraphics[width=0.44\linewidth]{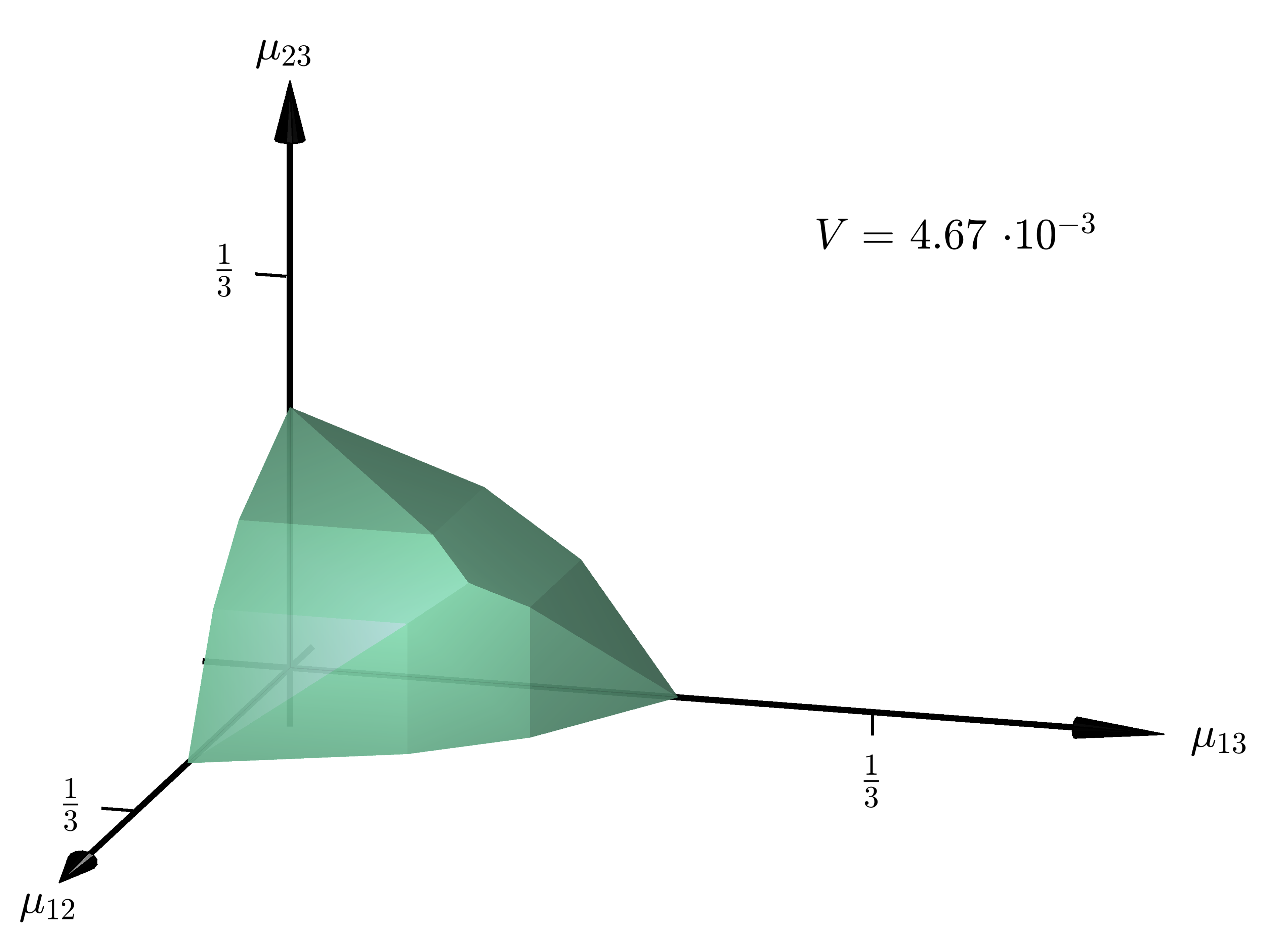}}
\hspace{1mm}
\subfloat{\includegraphics[width=0.44\linewidth]{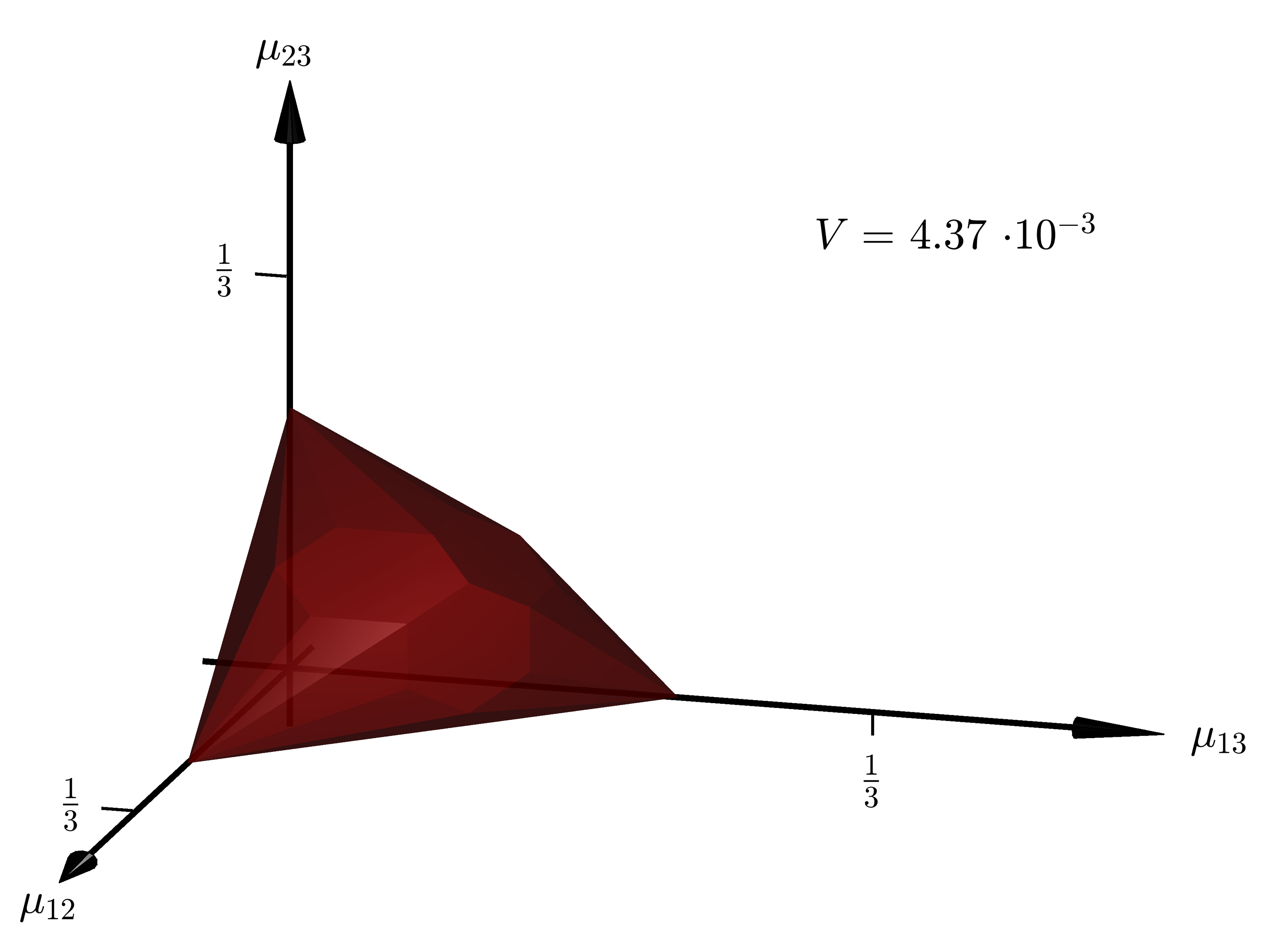}}
\hfill
\end{figure}
\begin{figure}[htbp]
\captionsetup[subfigure]{labelformat=empty}
\centering
\hspace{1mm}
\subfloat{\includegraphics[width=0.06\linewidth]{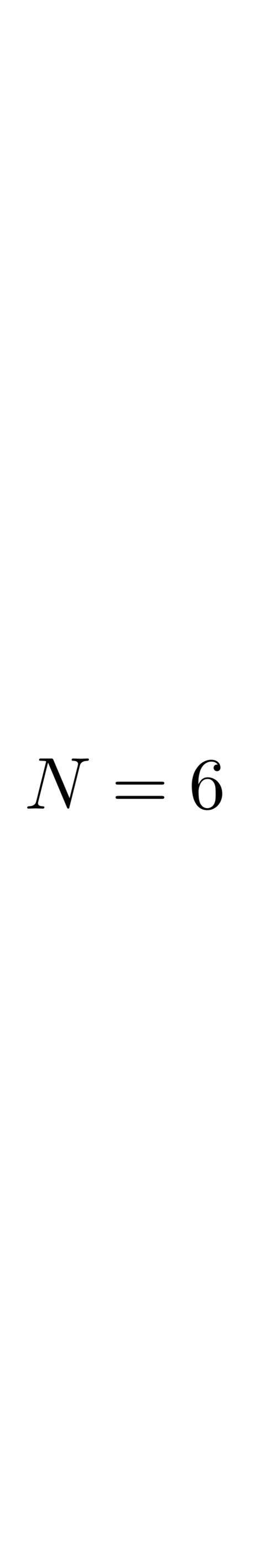}}
\hspace{1mm}
\subfloat{\includegraphics[width=0.44\linewidth]{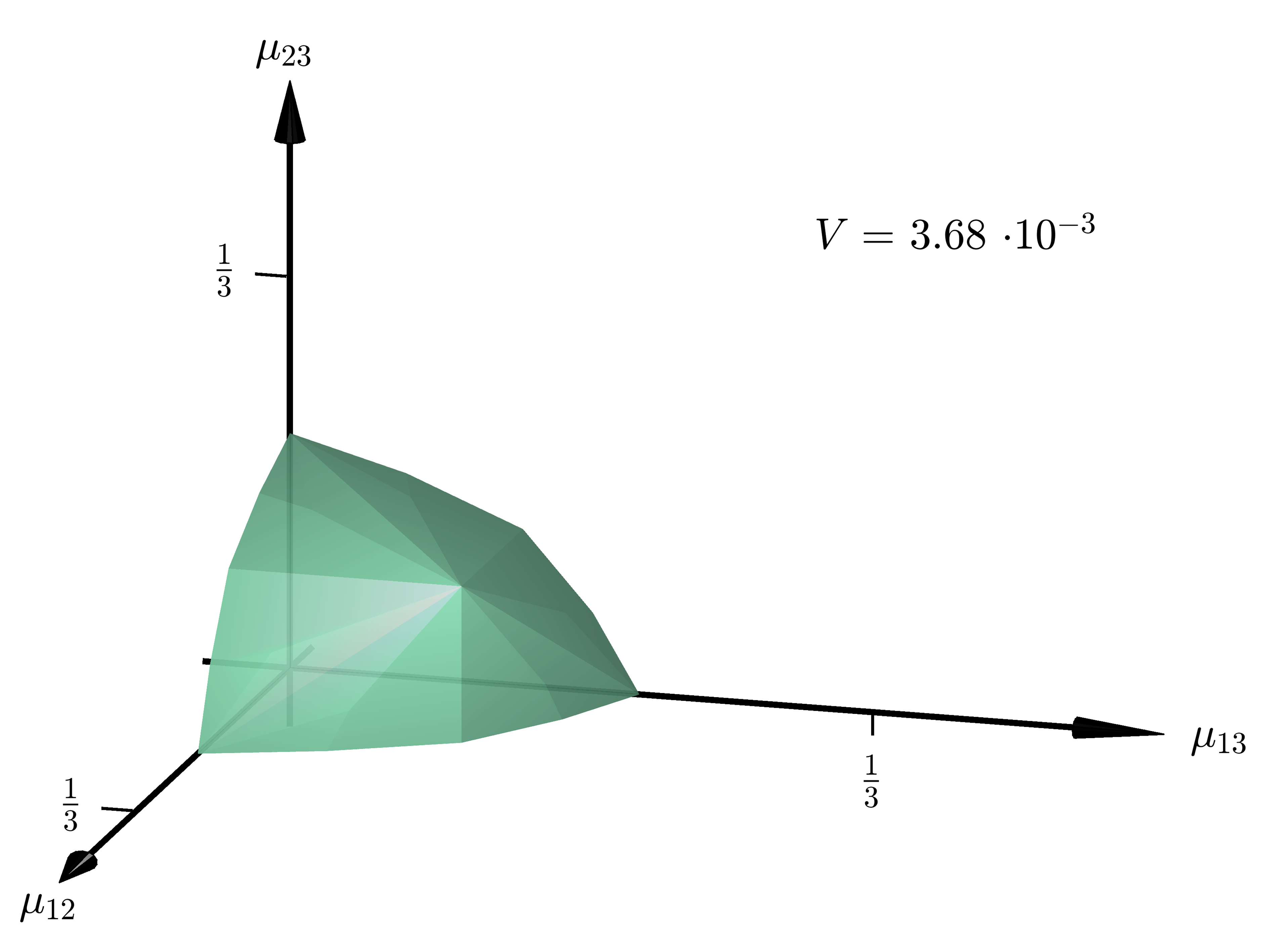}}
\hspace{1mm}
\subfloat{\includegraphics[width=0.44\linewidth]{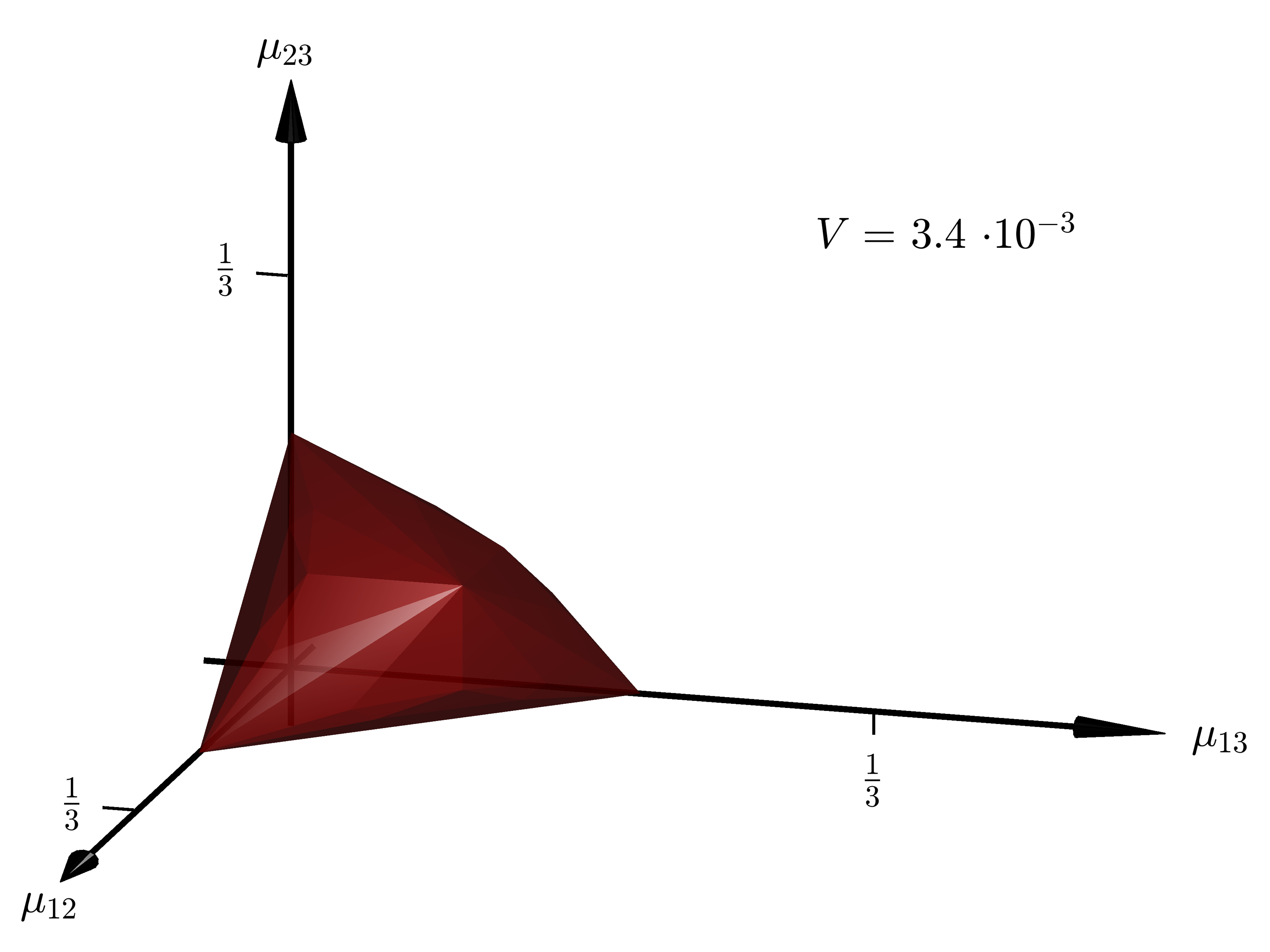}}
\hfill
\\
\hspace{1mm}
\subfloat{\includegraphics[width=0.06\linewidth]{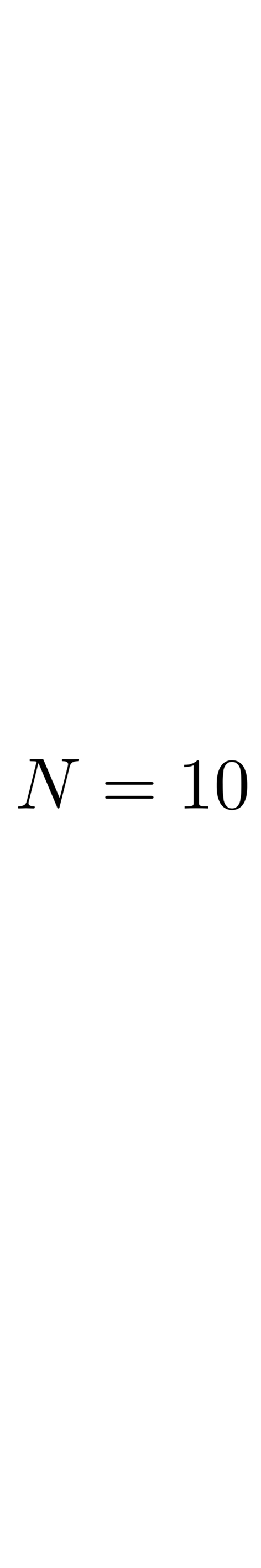}}
\hspace{1mm}
\subfloat{\includegraphics[width=0.44\linewidth]{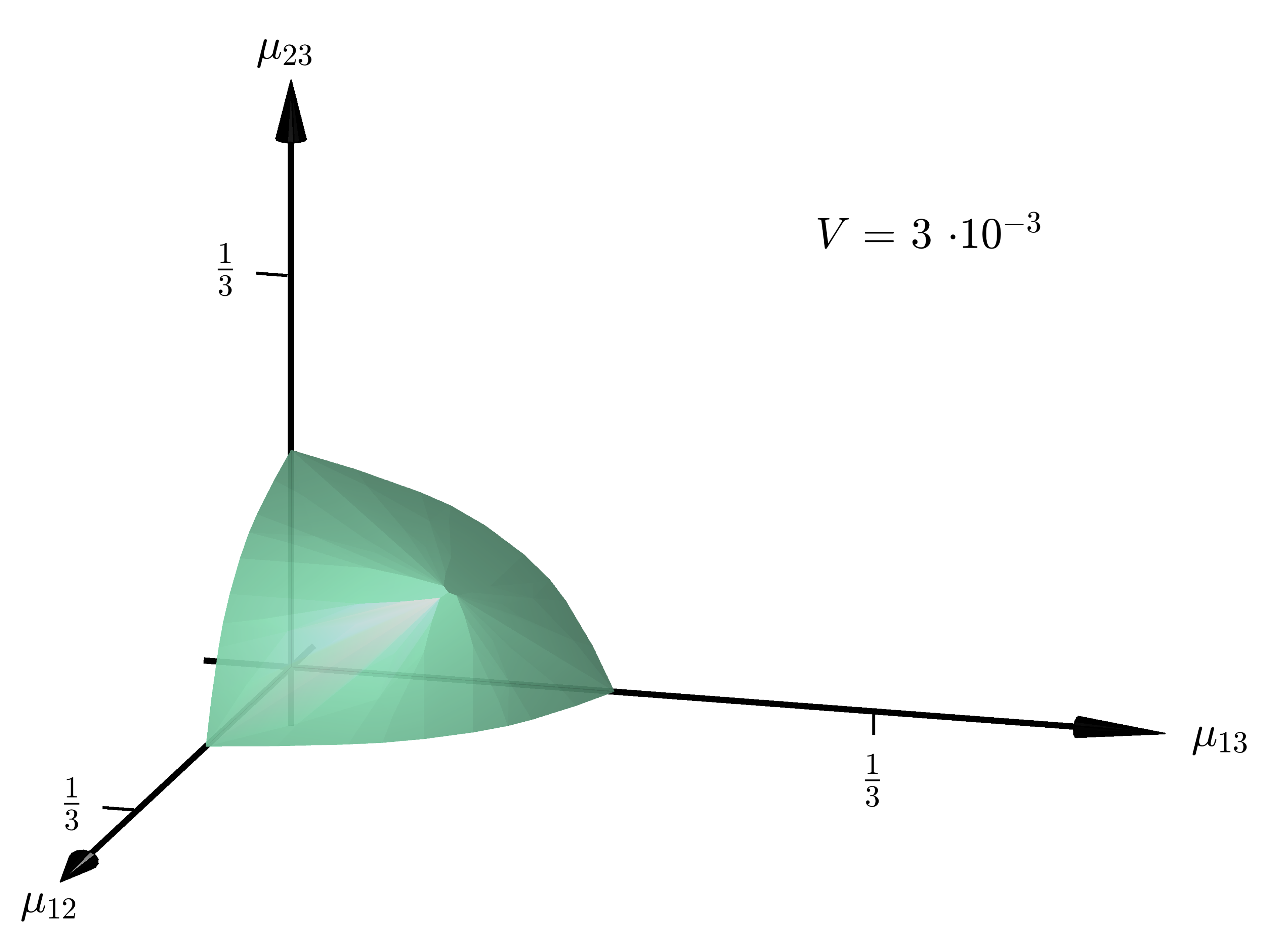}}
\hspace{1mm}
\subfloat{\includegraphics[width=0.44\linewidth]{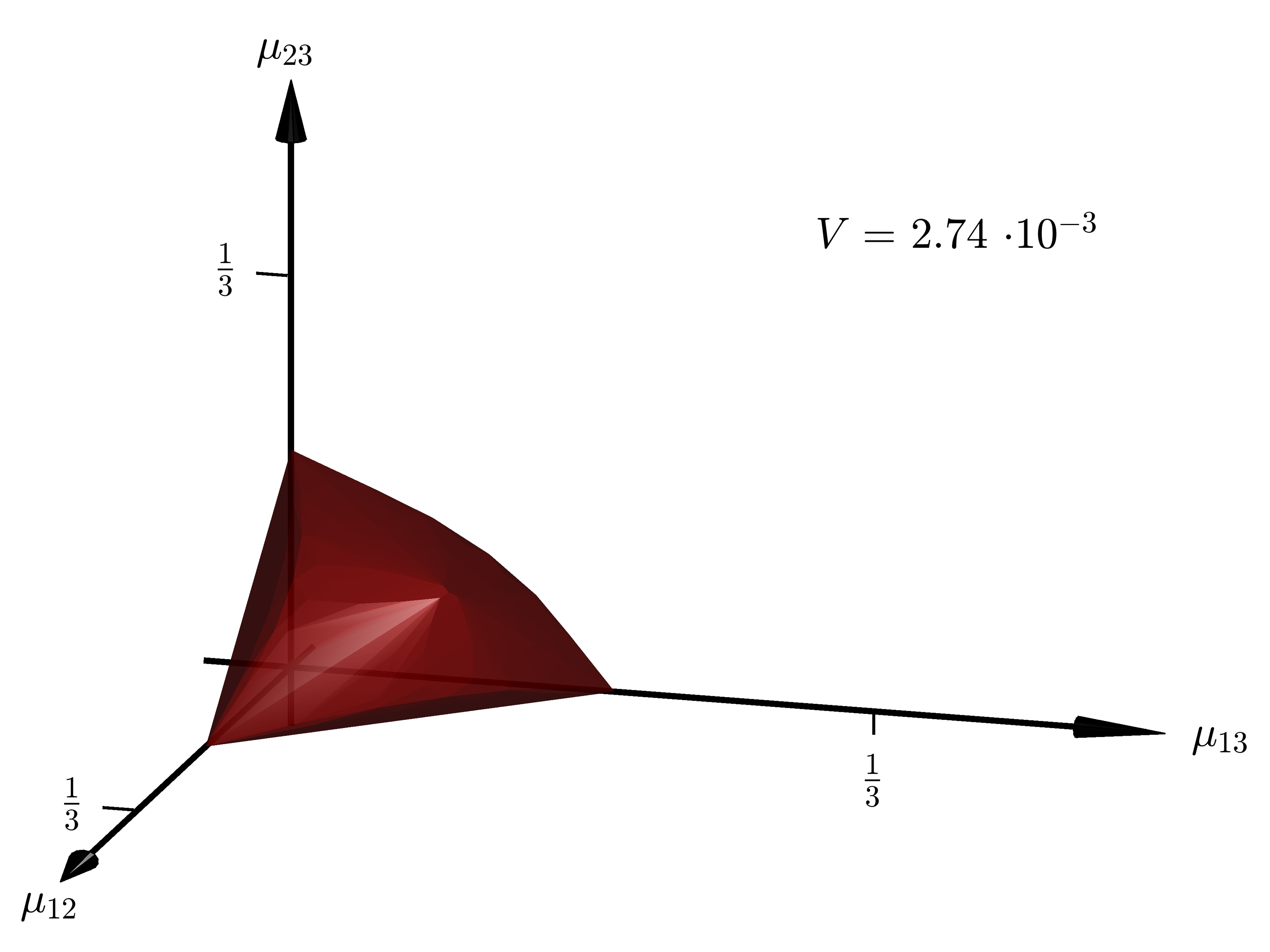}}
\hfill
\caption{The reduced Kantorovich respectively Monge polytope for $N$ marginals and $3$ states is visualized for $N=2,3,4,6$ and $10$ in green respectively red. The plots are arranged in ascending order with respect to $N$. The elements $\left(\mu_{ij}\right)_{i,j=1}^3$ of the polytopes are parametrized by their off-diagonal entries $\mu_{12},\mu_{13}$ and $\mu_{23}$. In the case of three marginals the reduced polytopes were initially depicted in \cite{GF18}. $V$ indicates the volume of the corresponding polytope. The volumetric ratio, reduced Monge polytope to reduced Kantorovich polytope, is depicted in Figure \ref{fig:Osc}.}
\label{fig:VisComp}
\end{figure}
\newline
\newline
\noindent Note that by definition the reduced Monge polytope is always contained in the reduced Kantorovich polytope independently of the number of marginals $N$ and the number of sites $\ell$. Some of the extreme points of the reduced Monge polytope are also extremal with respect to the reduced Kantorovich polytope and some lie on faces or in the interior of the latter (see Figure \ref{fig:ExtRed} for specific numbers). 
\newline
\newline
In the setting of $\ell=3$ sites, the numerical analysis of the reduced setting discussed in Section \ref{sec:IntroRed} yields that in the case of $N=2,3,\ldots,9$ and $10$ marginals there are always $5$ prominent extreme points of the reduced Kantorovich polytope that are of Monge-type. This is also indicated by Figure \ref{fig:KanMonge} as well as Figure \ref{fig:VisComp}. In the illustrations they can be identified with the four extreme points on the co-ordinate axes, including the origin, as well as the 'peak' in the front of the polytopes. In formulas these extreme points can be written as depicted in Table \ref{tab:ModProb}.
\begin{table}[htb]
\footnotesize
\centering 
\begin{tabular}{|c|c|l|c|}
\hline 
\multicolumn{2}{|c|}{Nomenclature} & \multicolumn{1}{|c|}{Abstract Notation}  & Matrix Notation \\ 
\hline 
\hline
\rule{0pt}{3.5ex}   
&  & $\frac{1}{3}M_2 \psi_N (\delta_1)$ & \multirow{3}{*}{$\begin{pmatrix}
\frac{1}{3} & 0 & 0 \\
 0 & \frac{1}{3} & 0 \\
0 & 0 & \frac{1}{3}  \\
\end{pmatrix}$} \\ 
$EA^{(N)}$& $EA^{(N)}$ & $ + \frac{1}{3}M_2 \psi_N (\delta_2) $ &  \\ 
&  & $+ \frac{1}{3} M_2 \psi_N (\delta_3)$ &  \\[6pt] 
\hline 
\rule{0pt}{7ex}   
& $ER^{(3m)}$ & $M_2 \psi_N \left( \frac{1}{3}\delta_1 + \frac{1}{3}\delta_2 + \frac{1}{3}\delta_3 \right) $& $\begin{pmatrix}
\frac{N-3}{9(N-1)} & \frac{N}{9(N-1)} & \frac{N}{9(N-1)} \\
\frac{N}{9(N-1)} & \frac{N-3}{9(N-1)}  & \frac{N}{9(N-1)} \\
\frac{N}{9(N-1)} & \frac{N}{9(N-1)} & \frac{N-3}{9(N-1)}  \\
\end{pmatrix}$ \\[26pt] 
&  & $\frac{1}{3} M_2 \psi_N \left( \frac{m}{N}\delta_1 + \frac{m}{N}\delta_2 + \frac{m+1}{N}\delta_3 \right) $ & \multirow{3}{*}{$\begin{pmatrix}
\frac{N-2}{9N} & \frac{N+1}{9N} & \frac{N+1}{9N} \\
\frac{N+1}{9N} & \frac{N-2}{9N}  & \frac{N+1}{9N} \\
\frac{N+1}{9N} & \frac{N+1}{9N} & \frac{N-2}{9N}  \\
\end{pmatrix}$} \\
$ER^{(N)}$  & $ER^{(3m+1)}$ & $+ \frac{1}{3} M_2 \psi_N \left( \frac{m}{N}\delta_1 + \frac{m+1}{N}\delta_2 + \frac{m}{N}\delta_3 \right) $ &  \\ 
&  & $ + \frac{1}{3} M_2 \psi_N \left( \frac{m+1}{N}\delta_1 + \frac{m}{N}\delta_2 + \frac{m}{N}\delta_3 \right) $ & \multirow{7}{*}{$\begin{pmatrix}
\frac{N-2}{9N} & \frac{N+1}{9N} & \frac{N+1}{9N} \\
\frac{N+1}{9N} & \frac{N-2}{9N}  & \frac{N+1}{9N} \\
\frac{N+1}{9N} & \frac{N+1}{9N} & \frac{N-2}{9N}  \\
\end{pmatrix}$} \\[12pt]
& & $\frac{1}{3} M_2 \psi_N \left( \frac{m}{N}\delta_1+ \frac{m+1}{N}\delta_2 + \frac{m+1}{N}\delta_3 \right) $ & \\
& $ER^{(3m+2)}$ & $ + \frac{1}{3} M_2 \psi_N \left( \frac{m+1}{N}\delta_1 + \frac{m}{N}\delta_2 + \frac{m+1}{N}\delta_3 \right)$ &  \\
& & $ + \frac{1}{3} M_2 \psi_N \left( \frac{m+1}{N}\delta_1 + \frac{m+1}{N}\delta_2 + \frac{m}{N}\delta_3 \right) $ & \\[8pt] 
\hline  
\rule{0pt}{3.5ex}  
& \multirow{3}{*}{$E12^{(2m)}$} & \multirow{3}{*}{$\frac{2}{3} M_2 \psi_N \left(\frac{1}{2} \delta_1 + \frac{1}{2} \delta_2 \right)+ \frac{1}{3} M_2 \psi_N \delta_3$}  & \multirow{3}{*}{$\begin{pmatrix}
\frac{N-2}{6(N-1)} & \frac{N}{6(N-1)} & 0 \\
\frac{N}{6(N-1)} & \frac{N-2}{6(N-1)} & 0 \\
0 & 0 & \frac{1}{3}  \\
\end{pmatrix}$} \\[19pt] 
$E12^{(N)}$ &  &  &  \\ 
 &  & $\frac{1}{3} M_2 \psi_N \left(\frac{N-1}{2N} \delta_1 + \frac{N+1}{2N} \delta_2 \right)$ & \multirow{3}{*}{$\begin{pmatrix}
\frac{N-1}{6N} & \frac{N+1}{6N} & 0 \\
\frac{N+1}{6N} & \frac{N-1}{6N} & 0 \\
0 & 0 & \frac{1}{3}  \\
\end{pmatrix}$} \\ 
& $E12^{(2m+1)}$ & $ + \frac{1}{3} M_2 \psi_N  \left(\frac{N+1}{2N} \delta_1 + \frac{N-1}{2N} \delta_2 \right)$  &  \\ 
&  & $+ \frac{1}{3}M_2 \psi_N(\delta_3)$ &  \\[6pt] 
\hline 
\end{tabular} 
\caption{The prominent extreme points $EA^{(N)}$, $ER^{(N)}$ and $E12^{(N)}$ of the reduced Kantorovich polytope are depicted in abstract and matrix notation. Hereby $N$ corresponds to the number of marginals and $m \in \mathbb{N}_0$ is a non-negative integer, allowing us to distinguish between the various cases regarding $N$. As all the coefficients in the 'Abstract Notation'-column  are integer multiples of $\frac{1}{3}$, these extreme points are of Monge-type.}
\label{tab:ModProb}
\end{table}
\noindent In Figure \ref{fig:VisComp}, $EA^{(N)}$ corresponds to the origin, $ER^{(N)}$ to the 'peak' and $E12^{(N)}$ to the non-origin extreme point on the $\mu_{12}$-axis. $E12^{(N)}$ assumes an exemplary role in Table \ref{tab:ModProb}. The corresponding extreme points on the $\mu_{13}$- respectively $\mu_{23}$-axis can be expressed analogously in abstract as well as matrix notation and will be denoted by $E13^{(N)}$ respectively $E23^{(N)}$.
\newline \newline
So far we know by numerical analysis that $EA^{(N)}$,  $ER^{(N)}$, $E12^{(N)}$, $E13^{(N)}$ and $E23^{(N)}$ are extreme points of the reduced Kantorovich polytope $\PNreplambda$ in the cases of $N = 2,3,\ldots, 9$ and $10$ marginals. One can prove that this holds true for a general number $N \geq 2$ of marginals. For $E12^{(N)}$, $E13^{(N)}$ and $E23^{(N)}$ one can show this by following the same approach taken in Remark \ref{rem:IntObsTwo} \ref{rem:IntObsTwoTwo}). In case of $EA^{(N)}$ and $ER^{(N)}$ it is an immediate consequence of Theorem \ref{the:ModProp}. In the following $d: X \times X \to \mathbb{R}$ will denote the discrete metric defined by 
 \begin{equation*}
 d(x,y):= \begin{cases} 1 & \textrm{ if } x \neq y \\ 0 & \textrm{ if } x = y. \end{cases}
 \end{equation*}
\begin{thm}
\label{the:ModProp}
We consider the reduced multi-marginal OT problem \eqref{eq:ProblemTwo} for $N \geq 2$ marginals and $\ell = 3$ sites. 
\begin{enumerate}[a)]
\item For the attractive cost function $d: X \times X \to \mathbb{R}$ the unique minimizer is given by $EA^{(N)}$.
\item  For the repulsive cost function $c_R: X \times X \to \mathbb{R}$ given by 
\begin{equation}
\label{eq:CostR}
c_R(x,y):= \begin{cases} \frac{1}{d(x,y)} & \textrm{ if } x \neq y \\
B & \textrm{ if } x = y
\end{cases}
\end{equation} 
for some constant $B>1$, the unique minimizer is given by $ER^{(N)}$.
\end{enumerate}
\end{thm}
\begin{proof} In the following the elements of the reduced Kantorovich polytope will always be interpreted as matrices. Along those lines $D$ respectively $CR$ corresponds to the matrix notation of $d$ respectively $c_R$, i.e., $D_{ij} := d(a_i,a_j) $ respectively $CR_{ij} := c_R(a_i,a_j)$ for $i,j \in \{1,2,\ldots,\ell\}$, and $\langle \cdot, \cdot \rangle$ denotes the standard matrix scalar product. 
\begin{enumerate}[a)]
\item Note that by non-negativity of the cost function $d$ the objective value of an arbitrary admissible state  $\mu \in \PNreplambda
$ is non-negative, i.e., $\langle D, \mu \rangle \geq 0$. As $EA^{(N)}$ is admissible and yields an objective value of $0$, i.e., $\langle D , EA^{(N)} \rangle = 0$, it is an optimizer of the corresponding problem \eqref{eq:ProblemTwo}. Positivity of $D$ in its off-diagonal entries and the marginal constraint ensure that $EA^{(N)}$ is the unique optimizer. 
\item To prove the second assertion, we drop the marginal constraint in a reformulated version of the considered problem \eqref{eq:ProblemTwo} and calculate the extremal elements of $\PNreptwo$, which solve the new optimization problem. There will be a unique convex combination of the optimal extreme points of $\PNreptwo$, namely $ER^{(N)}$, that lies in $\PNreplambda$. This state then corresponds to the unique minimizer of problem \eqref{eq:ProblemTwo}.
\newline
\newline
We consider the problem 
\begin{equation}
\label{eq:ModProbProofOne}
\min_{\mu \in \PNreplambda} \langle CR, \mu \rangle.
\end{equation}
Subsequently changing the objective function to $\langle {CR} - {\mathbb{1}} , \cdot \rangle$, where all the entries of ${\mathbb{1}} \in \mathbb{R}^{3 \times 3}$ are given by $1$, and plugging in the  marginal constraint allows us to reformulate \eqref{eq:ModProbProofOne} as follows 
\begin{equation}
\label{eq:ModProbProofTwo}
\max_{\mu \in \PNreplambda} \mu_{12} + \mu_{13} + \mu_{23}.
\end{equation} 
By dropping the marginal constraint, further restricting the admissible set to the extreme points \eqref{eq:ForTwoMar} of $\PNreptwo$ and rescaling the new admissible set by $N^2$ leads to the new optimization problem 
\begin{equation}
\label{eq:ModProbProofThree}
\max_{\lambda \in N \PQN} \lambda_1 \lambda_2 + \lambda_1 \lambda_3 + \lambda_2 \lambda_3 = \lambda_1 \lambda_2 + (N-\lambda_3) \lambda_3. 
\end{equation}
Assuming $\left( \lambda_1^*, \lambda_2^*, \lambda_3^* \right)$ is an optimizer of problem \eqref{eq:ModProbProofThree}, then elementary calculations show that $\left( \lambda_1^*, \lambda_2^* \right)$ fulfills 
\begin{equation}
\label{eq:ModProbProofFour}
\left( \lambda_1^*, \lambda_2^* \right) \in \begin{cases} \left\{ \left( \frac{r}{2}, \frac{r}{2} \right) \right\} & \textrm{ if } r \textrm{ is even} \\ \left\{ \left( \frac{r-1}{2}, \frac{r+1}{2} \right), \left( \frac{r+1}{2}, \frac{r-1}{2} \right) \right\} & \textrm{ if } r \textrm{ is odd,} \end{cases}
\end{equation}
where $r := N-\lambda_3^*$. Otherwise $\left( \lambda_1^*,\lambda_2^*,\lambda_3^*\right)$ would not be optimal. Note that \eqref{eq:ModProbProofThree} always admits a maximizer as $\PQN$ is finite. 
\newline
This allows us to identify problem \eqref{eq:ModProbProofThree} with the one-parameter optimization problem
\begin{equation}
\label{eq:ModProbProofFive}
\max\left\{ \max_{r\in 2\mathbb{N}_0,r\leq N} -\frac{3}{4}r^2+Nr, \max_{r\in 2\mathbb{N}_0+1, r\leq N} -\frac{3}{4} r^2 +Nr-\frac{1}{4}\right\}.
\end{equation}
Elementary calculations reveal that $r\in \{0,1,\ldots,N\}$ is optimal regarding \eqref{eq:ModProbProofFive} if and only if
\begin{equation*}
r\in \begin{cases}
\{2m\} & \textrm{ if } N=3m \textrm{ for } m \in \mathbb{N}_0\\
\{2m,2m+1\} & \textrm{ if } N=3m+1 \textrm{ for } m \in \mathbb{N}_0\\
\{2m+1,2m+2\} & \textrm{ if } N=3m+2 \textrm{ for } m \in \mathbb{N}_0. 
\end{cases}
\end{equation*}
It immediately follows that $\lambda \in N\PQN$ is optimal with respect to problem \eqref{eq:ModProbProofThree} if and only if 
\begin{equation}
\label{eq:ModProbProofSix}
\lambda \in \begin{cases}
N\left\{\left(\frac{m}{N},\frac{m}{N},\frac{m}{N}\right)\right\} &  \textrm{if } N=3m \textrm{ for } m \in \mathbb{N}_0\\
N\left\{ \left(\frac{m}{N},\frac{m}{N},\frac{m+1}{N}\right),\left(\frac{m}{N},\frac{m+1}{N},\frac{m}{N}\right),\left(\frac{m+1}{N},\frac{m}{N},\frac{m}{N}\right)\right\} & \textrm{if } N=3m+1 \textrm{ for } m \in \mathbb{N}_0\\
N\left\{\left(\frac{m}{N},\frac{m+1}{N},\frac{m+1}{N}\right),\left(\frac{m+1}{N},\frac{m}{N},\frac{m+1}{N}\right),\left(\frac{m+1}{N},\frac{m+1}{N},\frac{m}{N}\right)\right\} & \textrm{if } N=3m+2 \textrm{ for } m \in \mathbb{N}_0.
\end{cases}
\end{equation}
Recall that \eqref{eq:ForTwoMar} allows us (after dropping the factor $N$ in \eqref{eq:ModProbProofSix})  to identify the given maximizers with exactly those extremal elements of $\PNreptwo$ that maximize the sum of their off-diagonal entries. As by Minkowski's theorem any element of $\PNreplambda$ can be written as convex combination of the extreme points of $\PNreptwo$ and in any of the considered cases in \eqref{eq:ModProbProofSix} there are unique coefficients given by 
\begin{align*}
\alpha_1=1 & \hspace{1cm} \textrm{if } N=3m \textrm{ for } m \in \mathbb{N}_0\\
\alpha_1=\alpha_2=\alpha_3=\frac{1}{3} &\hspace{1cm} \textrm{else}
\end{align*}
 that allow us to write $\overline{\lambda}$ as a convex combination of the respective optimizers in \eqref{eq:ModProbProofSix}, it is easy to see that $ER^{(N)}$ as defined in Table \ref{tab:ModProb} is the unique optimizer of \eqref{eq:ModProbProofTwo} and thereby \eqref{eq:ModProbProofOne} for any $N\geq2$.
\end{enumerate}
\end{proof}
\begin{rem}
It is an immediate consequence of Theorem \ref{the:ModProp} that 
\begin{equation*}
\gamma_{GS}:=\frac{1}{3}S\delta_{11\ldots1}+\frac{1}{3}S\delta_{22\ldots2}+\frac{1}{3}S\delta_{33\ldots3}
\end{equation*}
respectively 
\begin{equation*}
\gamma_{C}:=\frac{1}{3}S\delta_{1\tau (1)\ldots \tau ^{(N-1)}(1)}+\frac{1}{3}S\delta_{2\tau (2)\ldots \tau ^{(N-1)}(2)}+\frac{1}{3}S\delta_{3\tau (3)\ldots \tau ^{(N-1)}(3)},
\end{equation*}
where $\tau : \{1,2,3\} \to \{1,2,3\} $ is the cyclic permutation defined by  $\tau (1) = 2$, $\tau (2) = 3$, $\tau (3) = 1$ and $\tau ^{(i)}$ denotes the $i$-th composition of $\tau $ with itself, is a solution to the OT problem \eqref{eq:ProblemN} for the Gangbo-\'{S}wi\k{e}ch cost function $c_{GS}: X^N \to \mathbb{R}$ defined by 
\begin{equation}
\label{eq:CostGS}
c_{GS} \left( x_1, \ldots, x_N \right):= \sum_{1 \leq i < j \leq N} d(x_i,x_j) 
\end{equation}
respectively the Coulomb cost function $c_C: X^N \to \mathbb{R}$ defined by 
\begin{equation*}
c_{C} \left( x_1, \ldots, x_N \right):= \sum_{1 \leq i < j \leq N} c_R(x_i,x_j).
\end{equation*}
Here \eqref{eq:CostGS} is a discretization of the pair-cost considered in \cite{GS98}. Note that one could replace $d(\cdot,\cdot)$ in \eqref{eq:CostGS} with $d(\cdot,\cdot)^p$ for any $p>1$ without changing $c_{GS}$. 
\end{rem}
\noindent Next, we examine the behavior of the sequences $\left( EA^{(N)} \right)_{N\geq 2}$, $\left(ER^{(N)} \right)_{N\geq 2}$, $\left(E12^{(N)} \right)_{N\geq 2}$, $\left(E13^{(N)} \right)_{N\geq 2}$ and $\left(E23^{(N)} \right)_{N\geq 2}$ for $N$ tending to $\infty$. Taking a look at the right column in Table \ref{tab:ModProb}, it is easy to see that the following holds true.
\begin{align*}
EA^{(N)} \xrightarrow[]{N \to \infty} \begin{pmatrix}
\frac{1}{3} & 0 & 0 \\
0 & \frac{1}{3} & 0 \\
0 & 0 & \frac{1}{3} 
\end{pmatrix} & =: EA^{(\infty)},
ER^{(N)}  \xrightarrow[]{N \to \infty}  \begin{pmatrix}
\frac{1}{9} & \frac{1}{9} & \frac{1}{9} \\
\frac{1}{9} & \frac{1}{9} & \frac{1}{9} \\
\frac{1}{9} & \frac{1}{9} & \frac{1}{9} 
\end{pmatrix} =: ER^{(\infty)}, \\
E12^{(N)} & \xrightarrow[]{N \to \infty}  \begin{pmatrix}
\frac{1}{6} & \frac{1}{6} & 0 \\
\frac{1}{6} & \frac{1}{6} & 0 \\
0 & 0 & \frac{1}{3} 
\end{pmatrix} =: E12^{(\infty)}
\end{align*}
Here $E12^{(\infty)}$ assumes again an exemplary role and $E13^{(\infty)}$ as well as $E23^{(\infty)}$ are defined in an analogous manner. One can express these 'limit extreme points' in a more probabilistic manner $EA^{(\infty)} =  \frac{1}{3} \delta_{11} + \frac{1}{3} \delta_{22} + \frac{1}{3} \delta_{33}$, $ER^{(\infty)} =  \left( \frac{1}{3} \delta_{1} + \frac{1}{3} \delta_{2} + \frac{1}{3} \delta_{3} \right) \otimes \left( \frac{1}{3} \delta_{1} + \frac{1}{3} \delta_{2} + \frac{1}{3} \delta_{3} \right) $ as well as $E12^{(\infty)} =  \frac{2}{3}\left( \frac{1}{2} \delta_{1} + \frac{1}{2} \delta_{2} \right) \otimes \left( \frac{1}{2} \delta_{1} + \frac{1}{2} \delta_{2}  \right) + \frac{1}{3} \delta_{3}$ corresponding to the 'Abstract Notation'-column in Table \ref{tab:ModProb} via \eqref{eq:ForTwoMar}. 
\newline
\newline
In the following, $\mathcal{D}_{\infty\textrm{-rep},\overline{\lambda}}$ will denote the convex hull of these 'limit extreme points', i.e., \begin{equation}
\label{eq:Diamond}
\mathcal{D}_{\infty\textrm{-rep},\overline{\lambda}} = \conv \left( \left\{EA^{(\infty)}, ER^{(\infty)}, E12^{(\infty)}, E13^{(\infty)}, E23^{(\infty)}\right\} \right).
\end{equation}
For an illustration of $\mathcal{D}_{\infty\textrm{-rep},\overline{\lambda}}$ see Figure \ref{fig:Diamond}.
\newline
\begin{figure}[htbp]
\centering
\begin{minipage}[t][95mm][t]{0.48\textwidth}
  \centering
  \includegraphics[width=1.0\linewidth]{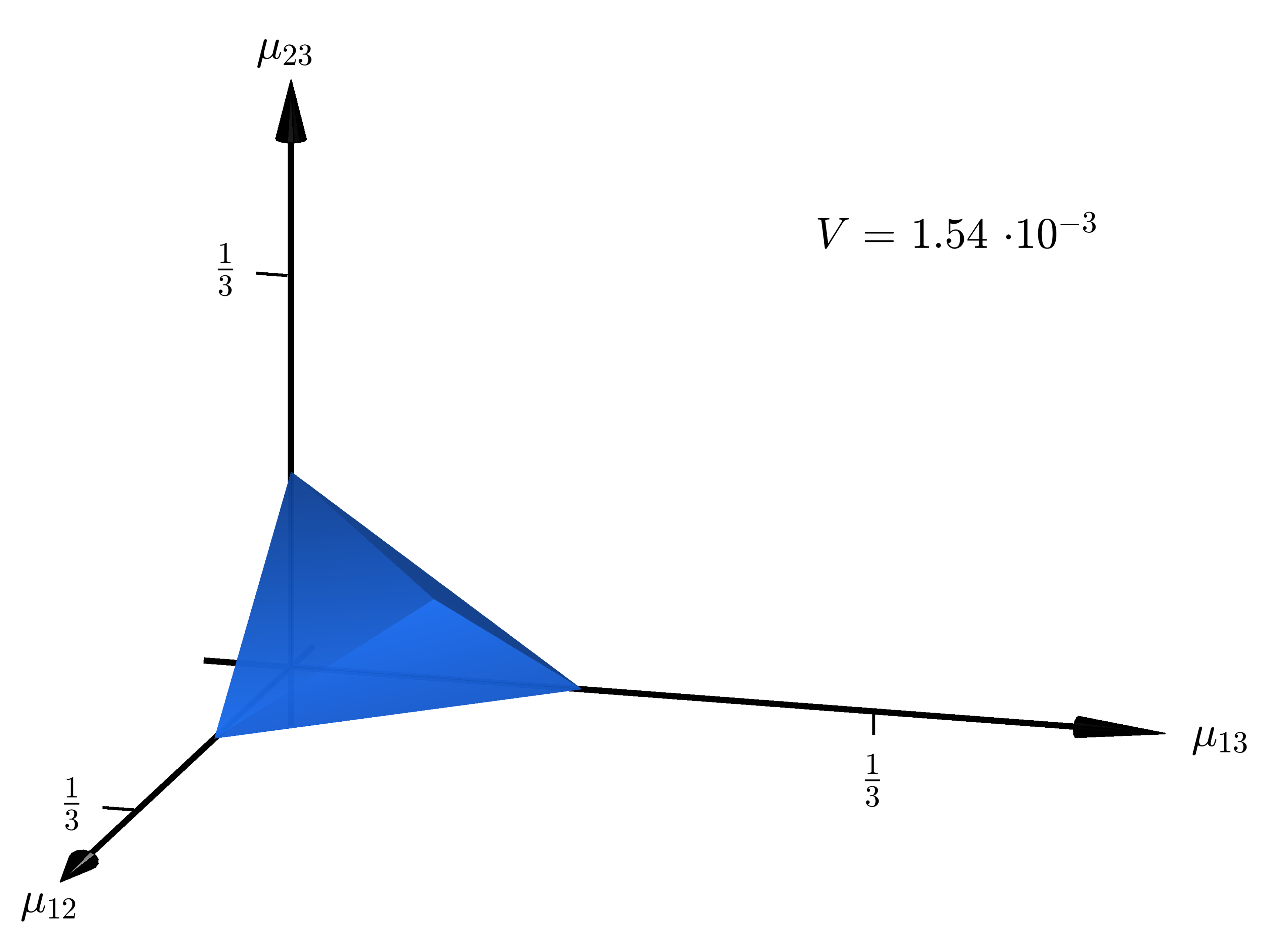}
 \captionof{figure}{The diamond-shaped polytope $\mathcal{D}_{\infty\textrm{-rep},\overline{\lambda}}$, as defined in \eqref{eq:Diamond}, is depicted in blue. The elements $(\mu_{ij})^3_{i,j=1}$ of the polytope are parametrized by their off-diagonal entries $\mu_{12},\mu_{13}$ and $\mu_{23}$. The volume of $\mathcal{D}_{\infty\textrm{-rep},\overline{\lambda}}$ is indicated in the upper-right corner.}
  \label{fig:Diamond}
\end{minipage}%
\hfill
\begin{minipage}[t][95mm][t]{.48\textwidth}
  \centering
  \includegraphics[width=1.0\linewidth]{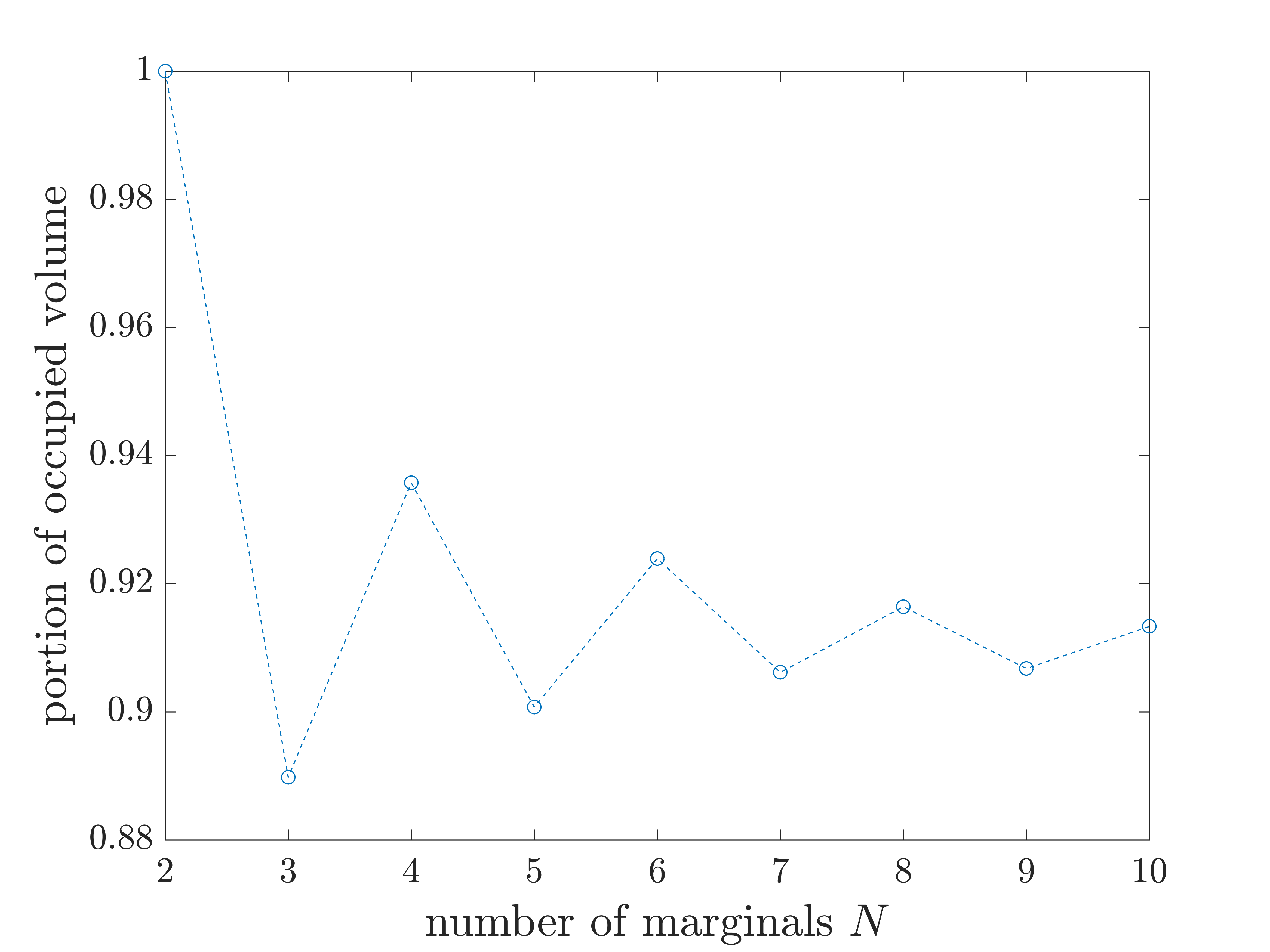}
  \captionof{figure}{The volumetric ratio, reduced Monge polytope to reduced Kantorovich polytope for $N$ marginals and $3$ states, is depicted in dependency of the number of marginals $N$.}
 \label{fig:Osc}
\end{minipage}
\end{figure}
\newline
It was proven in \cite{FMPCK13} that $N$-representability becomes an increasingly stringent condition as $N$ grows, in more detail, $\PNreptwo \subseteq \mathcal{P}_{\hat{N}\textrm{-rep}}\left(X^2\right) $ for any $N \geq \hat{N} \geq 2$. It follows immediately that the reduced Kantorovich polytope for $N$ marginals and $3$ states $\PNreplambda$ is contained in the reduced Kantorovich polytope $\mathcal{P}_{\hat{N}\textrm{-rep},\overline{\lambda}}\left(X^2\right)$ for $\hat{N}$ marginals and $3$ states. As $\PNreplambda$ is closed and convex, $\mathcal{D}_{\infty\textrm{-rep},\overline{\lambda}}$ is a subset of the reduced Kantorovich polytope $\PNreplambda$ for any number $N \geq 2$ of marginals and $3$ sites. In summary we get the following chain of inequalities 
\begin{equation}
\label{eq:ChainofIne}
 \min_{\mu \in \PNreplambda} V[\mu]  \leq \min_{\mu \in \mathcal{P}_{(N+1)\textrm{-rep},\overline{\lambda}}\left(X^2\right)} V[\mu] \leq \dots \leq \min_{\mu \in \mathcal{D}_{\infty\textrm{-rep},\overline{\lambda}}} V[\mu] \leq V[ER^{(\infty)}] ,
\end{equation}
where  $V[\mu] := \int_{X^2} v(x,y)  d\mu(x,y)$.
The inequalities \eqref{eq:ChainofIne} show that for any number of marginals $N \geq 2$ we can find an upper bound of the optimal value in \eqref{eq:ProblemTwo} by computing the objective value of the 'attractive limit extreme point' $EA^{(\infty)}$, the 'repulsive limit extreme point' $ER^{(\infty)}$ and the 'axis limit extreme points' $E12^{(\infty)}$, $E13^{(\infty)}$ as well as $E23^{(\infty)}$ and choosing the smallest one. Note that this improves the in physics common mean field approximation $V[ER^{(\infty)}]$, where one usually considers repulsive pair-costs $v: X \times X \to \mathbb{R}$. 
\newline
\newline 
Finally, we note that the volume portion of the reduced Kantorovich polytope that is occupied by the reduced Monge polytope exhibits oscillatory behavior with decreasing amplitude when interpreted as a function of the number of marginals $N$, see Figure \ref{fig:Osc}. The considered volumetric ratio oscillates around a value above 0.9 where even marginals when directly compared to the odd marginals produce a higher ratio. In the sense that an optimizer in the occupied volume yields the existence of a Monge-solution, the Monge ansatz seems to be 'better' for an even number of marginals. 
\section{Lower Bound on Extremal Coefficients}
\label{sec:ExtremalCoefficients}
The results of this section were achieved in the pursuit of a generalization of Theorem \ref{the:ModProp}. When replacing $c_R$ in \eqref{eq:CostR} with a general repulsive interaction, the optimization process no longer boils down to a maximization of the sum of off-diagonal entries. Even subtle differences in the off-diagonal cost coefficients could influence the optimization. Still, our intuition tells us, that - if the repulsion is strong enough - it is still best to distribute the $\frac{1}{N}$-quantized entries as uniformly as possible among the given $\ell$ sites. But what is the reason behind the non-existence of cases for which it is best to attain an unevenly distributed configuration with a very small probability? The answer is given in the following theorem.
\begin{thm}[lower bound on extremal coefficients]
\label{the:LB}
Assume $\alpha$ to be an extreme point of the polytope $\Pcoef$, as defined in \eqref{eq:Pcoef}. Then each nonzero entry of $\alpha$ is bigger than or equal to $\frac{1}{\ell N^{\ell-1}}$, i.e., for all $\nu\in\left\{1,\ldots, {N+\ell-1\choose N} \right\}$
\begin{equation*}
\alpha_{\nu} \neq 0 \rightarrow \alpha_{\nu} \geq \frac{1}{\ell N ^{\ell-1}}.
\end{equation*}
\end{thm}
\begin{proof}
Let $\alpha$ be an arbitrary extreme point of the polytope $\Pcoef$. Then - as discussed in Section \ref{sec:ClassKan} - the nonzero entries of $\alpha$ indicate a selection of columns of $A$ that are linearly independent. \\
In case the cardinality of this selection is strictly less than $\ell$, we add a suitable choice of elements of $\PQN$ in order to form a basis $B$ of $\mathbb{R}^\ell$. Otherwise the present selection already constitutes such a basis $B$. In the following $A_B$ will denoted the $\ell \times \ell$-submatrix of $A$ that consists of exactly those columns that are contained in $B$, accordingly $\alpha_B$ is the subvector of $\alpha$ that is reduced to those entries that correspond to elements of $B$. By construction of $B$, 
\begin{equation}
\label{eq:InvertibleEquaion}
A_B\alpha_B = \overline{\lambda}
\end{equation}
holds. As $A_B$ is invertible, \eqref{eq:InvertibleEquaion} is equivalent to 
\begin{equation}
\label{eq:InvertibleEquationTwo}
\alpha_B=A_B^{-1}\overline{\lambda}=\frac{1}{\mathrm{det}(A_B)} C^T \overline{\lambda}
\end{equation}
with $C$ denoting the cofactor matrix of $A_B$. Each entry of $C$ is the product of a sign factor and the determinant of an $(\ell-1)\times (\ell-1)$-submatrix of $A_B$. As the columns of $A_B$ represent $\frac{1}{N}$-quantized probability measures, the entries of $A_B$ are integer multiples of $\frac{1}{N}$ and therefore the entries of $C$ are integer multiples of $\frac{1}{N^{\ell-1}}$. Utilizing Hadamard's inequality, one easily sees that $|\mathrm{det}(A_B)|\leq 1$ holds. Finally we recall that each of the entries of $\overline{\lambda}$ is given by $\frac{1}{\ell}$. Consequently, for the $i$-th entry $(\alpha_B)_i$ of $\alpha_B$ (with $i\in\{1,\ldots,\ell\}$) the following holds
\begin{equation*}
(\alpha_B)_i=|(\alpha_B)_i|=\frac{1}{|\mathrm{det}(A_B)|}|(C^T\overline{\lambda})_i|\geq |(C^T\overline{\lambda})_i|=\frac{k_i}{\ell N^{\ell-1}}
\end{equation*}
with $k_i$ being some non-negative integer. As $k_i$ is zero if and only if $(\alpha_B)_i$ is zero the proof of Theorem \ref{the:LB} is complete. 
\end{proof}
\noindent We do not expect the lower bound established in Theorem \ref{the:LB} to be sharp. The key to unlocking an improvement and even potentially a quantization of the extremal coefficients lies in a deeper analysis of \eqref{eq:InvertibleEquationTwo}, particulary the entries of the cofactor matrix $C$ and how they relate to one another as well as the determinant of $A_B$ itself. This analysis, however, lies beyond the scope of this paper and we consider it to be subject to future research.  
\\
We consider Theorem \ref{the:LB} to be the bedrock the upcoming theorem is built upon.
\begin{thm}[support of optimal couplings regarding 'repulsive OT problems'] 
\label{the:SupOptC}
Let $N\geq 2$ be the number of marginals and $\ell\geq 2$ the number of states. Consider the OT problem
\begin{multline}
\label{eq:ThmProbFor}
\mathrm{Minimize} \int_{X^N} \sum_{1\leq i<j \leq N} v(x_i,x_j)\mathrm{d}\gamma(x_1,\ldots,x_N) \\
\mathrm{\ over \ } \gamma \in \Psym \mathrm{\ subject \ to \ } \gamma \mapsto \overline{\lambda},
\end{multline} 
with $v:X^2\to \mathbb{R}$ being a symmetric pair-potential that fulfills $v(x,x)=B$ for all $x\in X$ and some constant
\begin{equation*}
B>\left(\max_{x_i \neq x_j} v(x_i,x_j)\right) + N^{\ell+1}\ell \left(\left(\max_{x_i \neq x_j} v(x_i,x_j)\right)-\left(\min_{x_i \neq x_j} v(x_i,x_j)\right)\right).
\end{equation*}
Let $\gamma$ be an optimizer of the considered problem. If $\gamma$ gives mass to a point $(x_1,\ldots,x_N)\in X^N$ then each state $a \in X$ appears either $\left\lfloor\frac{N}{ \ell }\right\rfloor$ or $\left\lceil\frac{ N }{\ell}\right\rceil$ times in the given tupel, i.e., $|\{i:x_i=a\}|\in \left\{ \left\lfloor\frac{N}{ \ell }\right\rfloor, \left\lceil\frac{ N }{\ell}\right\rceil \right\}$ for all $a\in X=\{a_1,\ldots,a_\ell\}$. \\
(Note that in case $\frac{N}{\ell}=k$ for some $k\in \mathbb{N}$ each state appears exactly $k$ times.)
\end{thm}
\begin{proof}
Based on the proof of Theorem \ref{the:ModProp} assertion b) we start off by rewriting the objective function, i.e., the function that is to be minimized. Successively using the pairwise symmetric structure of the cost function, identifying the function $v:X^2 \to \mathbb{R}$ as well as the measure $M_2 \gamma$ on $X^2$ with their respective matrix counterparts $(v_{ij})^\ell_{i,j=1}, \left((M_2 \gamma)_{ij}\right)^\ell_{i,j=1}$ with $v_{ij}:=v(a_i,a_j), (M_2 \gamma)_{ij}:=M_2 \gamma(\{(a_i,a_j)\})$ and finally utilizing the marginal constraint allows us to write the objective value of any admissible $\gamma$ independent of its diagonal-entries.
\begin{align*}
\int_{X^N}\sum_{1\leq i < j \leq N} v(x_i,x_j)\mathrm{d}\gamma (x_1,\ldots,x_N) & = {N \choose 2} \int_{X^2}v(x,y)\mathrm{d}(M_2 \gamma)(x,y)\\
& = {N \choose 2} \sum^\ell_{i,j=1}v_{ij}(M_2 \gamma)_{ij}\\
& = {N \choose 2} \Bigg( B + \sum^\ell_{\substack{i,j=1 \\ i \neq j}} (v_{ij}-B) (M_2 \gamma)_{ij}\Bigg)\\
& = {N \choose 2}B + {N \choose 2} \sum^\ell_{\substack{i,j=1 \\ i \neq j}} (v_{ij}-B) (M_2 \gamma)_{ij}
\end{align*}
Now one easily sees that 
\begin{multline}
\label{eq:ProofProbFor}
\mathrm{Maximize} \ C_v[\gamma]:= \sum^\ell_{\substack{i,j=1 \\ i \neq j}} (B-v_{ij}) (M_2 \gamma)_{ij} \\
\mathrm{over} \ \gamma\in \Psym \mathrm{ \ subject \ to \ } \gamma \mapsto \overline{\lambda},
\end{multline}
is an equivalent problem formulation, in the sense that any admissible $\gamma$ is optimal with respect to \eqref{eq:ThmProbFor} if and only if it is optimal with respect to the problem at hand. \\
\\
Now let $\gamma$ be such an optimizer solving the problem stated in the considered theorem \eqref{eq:ThmProbFor} as well as the problem given by \eqref{eq:ProofProbFor}. Firstly, we assume $\gamma$ to be an extreme point of the symmetric Kantorovich polytope for $N$ marginals and $\ell$ states, i.e., the set of admissible trial states. Then by Corollary \ref{cor:NewFor} the to $\gamma$ corresponding coefficient vector $\alpha=R^{-1}\gamma$ is itself an extreme point of the polytope $\Pcoef$. $\alpha$ then fulfills 
\begin{equation}
\label{eq:CCGamma}
\gamma=R\alpha=\sum_{\lambda\in \mathcal{P}_{\frac{1}{N}}(X)}\alpha_\lambda \psi_N (\lambda).
\end{equation}
Recall that $\psi_N(\lambda)$ denotes the (uniquely determined) symmetrized Dirac-measure \eqref{eq:ExtrSym} with one-point marginal $\lambda$. Readers feeling lost regarding the present notation are advised to take a look back at Section \ref{sec:ClassKan}, particularly pages 7-8. By linearity of $C_v[\cdot]$ it holds
\begin{equation*}
C_v[\gamma]=\sum_{\lambda \in \mathcal{P}_{\frac{1}{N}}(X)}\alpha_\lambda C_v [\psi_N (\lambda)].
\end{equation*}
As already stated in \eqref{eq:ForTwoMar}, $M_2 \psi_N(\lambda)=\frac{N}{N-1}\lambda \otimes \lambda - \frac{1}{N-1}(\mathrm{id,id})\#\lambda $ holds yielding 
\begin{equation*}
C_v[\psi_N (\lambda)]=\frac{N}{N-1}\sum^\ell_{\substack{i,j=1 \\ i \neq j}}(B-v_{ij})\lambda_i \lambda_j.
\end{equation*}
In the following we will take a closer look at the objective value $C_v[\cdot]$ of the extremal symmetric probability measures $\psi_N(\lambda)$ by investigating the behaviour of the function $f:\mathcal{P}_{\frac{1}{N}}(X)\to \mathbb{R}$ given by 
\begin{equation*}
f(\lambda) = C_v [\psi_N(\lambda)].
\end{equation*}
We split $f$ into a dominant and a submissive term, denoted by $d$ and $s$, respectively. 
\begin{equation*}
f(\lambda) = d(\lambda) + s(\lambda)
\end{equation*}
\begin{equation*}
\mathrm{with} \ d(\lambda):=\frac{N}{N-1}\sum^{\ell}_{\substack{i,j=1 \\ i \neq j}}(B-v_{i^*j^*})\lambda_i\lambda_j \mathrm{\ and \ } s(\lambda):=\frac{N}{N-1}\sum^{\ell}_{\substack{i,j=1 \\ i \neq j }} (v_{i^* j^*}-v_{ij})\lambda_i \lambda_j,
\end{equation*}
\begin{equation*}
 \mathrm{where \ } i^*,j^* \mathrm{ \ are \ indices \ fulfilling \ } v_{i^*j^*}=\max_{\substack{i,j \\ i \neq j}}v_{ij}. 
 \end{equation*}
 Hereby, a more compact manner to write $d$ is given by 
\begin{equation*}
d(\lambda)=\lambda^TD\lambda \ \ \mathrm{for } \ \ D=(D_{ij})\in \mathbb{R}^{\ell \times \ell} \ \mathrm{defined \ by \ } D_{ij} = \begin{cases} 0 & \mathrm{for} \ i=j \\ B' & \mathrm{for} \ i \neq j \end{cases}
\end{equation*}
\begin{equation}
\label{eq:DefBPrime}
\mathrm{with \ } B':= \frac{N}{N-1}(B-v_{i^*j^*}).
\end{equation}
By Taylor-expanding $d$ at the uniform probability measure $\overline{\lambda}$ and utilizing the geometry of the matrix $D$ embodied by its eigenspaces one easily sees that 
\begin{equation*}
d(\lambda)=d(\overline{\lambda})-B'|\lambda-\overline{\lambda}|^2
\end{equation*}
with $|\cdot|$ denoting the Euclidean norm in $\mathbb{R}^\ell$. Now elementary arguments and calculations reveal the following. $\hat{\lambda}$ maximizes $d(\cdot)$ among the $\frac{1}{N}$-quantized probability measures $\mathcal{P}_{\frac{1}{N}}(X)$ if and only if $r$ entries of $\hat{\lambda}$ are given by $\frac{m+1}{N}$ and the remaining $(\ell-r)$ entries correspond to $\frac{m}{N}$ with $m:=\left\lfloor\frac{N}{ \ell }\right\rfloor$ and $r:= N-m\ell$. Any deviating $\lambda\in \PQN$, i.e., any $\lambda$ not obeying these restrictions regarding its entries, decreases the value of $d(\cdot)$ by at least $\frac{B'}{N^2}$. That is, for a rule-abiding $\hat{\lambda}$ and a deviating $\lambda$ it holds, $d(\hat{\lambda})-d(\lambda)\geq \frac{B'}{N^2}$.\\
 \\
Now we return to the consideration of $\gamma$ which is assumed to be a solution of the problems \eqref{eq:ThmProbFor} and \eqref{eq:ProofProbFor} as well as (for now) an extreme point of both sets of admissible trial states. $\alpha$ denotes the coefficient vector underlying the representation of $\gamma$ as a convex combination of extremal symmetric probability measures as given in \eqref{eq:CCGamma}. As already mentioned above, $\alpha$ is itself an extreme point of $\Pcoef$.\\
The next step is to derive a contradiction starting from the assumption $\alpha_\lambda>0$ for a $\lambda\in \PQN$ that deviates from the 'entry laws' described above. In the following this deviating $\frac{1}{N}$-quantized probability measure will be denoted by $\tilde{\lambda}$. Let further $\hat{\gamma}= \sum_{\lambda\in \PQN}\hat{\alpha}_\lambda \psi_N (\lambda) $ be an admissible trial state whose coefficient vector $\hat{\alpha}$ only gives mass to law-abiding $\lambda \mathrm{s} \in \PQN$. It is easy to see that such a state always exists. Then - with $i_*,j_*$ denoting indices that fulfill $v_{i_* j_*}= \min_{\substack{i,j \\ i \neq j}}v_{ij}$ and $\hat{\lambda}$ denoting an arbitrary law-abiding $\lambda \in \PQN$ - it holds 
\begin{align*}
C_v[\hat{\gamma}]-C_v[\gamma] & = \sum_{\lambda\in \PQN} \hat{\alpha}_\lambda (d(\lambda)+s(\lambda))-\sum_{\lambda\in \PQN}\alpha_\lambda (d(\lambda)+s(\lambda))\\
&= \sum_{\lambda\in \PQN}(\hat{\alpha}_\lambda- \alpha_\lambda)s(\lambda)+ \sum_{\lambda\in \PQN} \hat{\alpha}_\lambda d(\lambda) - \sum_{\lambda\in \PQN} \alpha_\lambda d(\lambda)\\
& \geq - \frac{N}{N-1}(v_{i^* j^*}-v_{i_* j_*})+d(\hat{\lambda}) - \frac{1}{\ell N^{\ell-1}}d(\tilde{\lambda})-\left(1-\frac{1}{\ell N^{\ell-1}}\right) d(\hat{\lambda})\\
&= - \frac{N}{N-1}(v_{i^* j^*}-v_{i_* j_*})+ \frac{1}{\ell N^{\ell-1}}(d(\hat{\lambda}) - d(\tilde{\lambda}))\\
&\geq - \frac{N}{N-1}(v_{i^* j^*}-v_{i_* j_*})+ \frac{B'}{\ell N^{ \ell+1}}.
\end{align*}
Hereby, elementary estimates gave us a lower bound on how much $\hat{\gamma}$ might loose to $\gamma$ regarding the submissive function $s$. We further took advantage of the fact that all law-abiding $\lambda$s produce the same dominant value $d(\hat{\lambda})$. The key step, however, is to establish an upper-bound on the portion of the objective value of $\gamma$ which corresponds to the dominant function $d$. This bound is based on Theorem \ref{the:LB} as well as the priorly established fact that the law-abiding elements of $\PQN$ are exactly those that maximize $d(\cdot)$ and $\tilde{\lambda}$ falls short by at least $\frac{B'}{N^2}$ in comparison. 
\\
Combining the definition of $B'$ in \eqref{eq:DefBPrime} with the assumption on $B$ now yields $C[\hat{\gamma}]>C[\gamma]$. This finalizes the contradiction. Hence, positivity of a coefficient $\alpha_\lambda$ implies that $\lambda$ fulfills the 'entry laws'. Recalling the representation \eqref{eq:CCGamma} of $\gamma$ as well as the definition of $\psi_N$ now reveals that the statement of Theorem \ref{the:SupOptC} is true for the extremal $\gamma$. \\
As any non-extremal optimizer may be written as a convex combination of extremal ones the proof of Theorem \ref{the:SupOptC} is complete.
\end{proof}
\noindent The following consequence of Theorem \ref{the:SupOptC} is a generalization of Theorem \ref{the:ModProp} assertion b). 
\begin{coro}
\label{cor:OptPeakGen}
We consider the reduced multi-marginal OT problem \eqref{eq:ProblemTwo} for $N\geq 2$ marginals and $\ell=3$ sites. \\
For any symmetric cost function $v:X\times X \to \mathbb{R}$ that fulfills $v(x,x)=B$ for all $x\in X$ and some constant 
\begin{equation*}
B>\left(\max_{x_i \neq x_j} v(x_i,x_j)\right) + N^{\ell+1}\ell \left(\left(\max_{x_i \neq x_j} v(x_i,x_j)\right)-\left(\min_{x_i \neq x_j} v(x_i,x_j)\right)\right)
\end{equation*}
the unique minimizer is given by $ER^{(N)}$.
\end{coro}
\begin{proof}
We start off with a change of venue and consider the 'unreduced' problem version \eqref{eq:ThmProbFor}. With the number of states $\ell$ being equal to three, Theorem \ref{the:SupOptC} reduces the points an optimizer might give mass to already to such an extent that the optimizer's uniqueness follows. The two-point marginal of said optimizer is given by $ER^{(N)}$ which inherits the status of a unique optimizer from its representing measure. 
\end{proof}
\noindent Recall that $ER^{(N)}$ is of Monge-type. Consequently, Corollary \ref{cor:OptPeakGen} provides a class of repulsive costs yielding a unique Monge optimizer. All of these examples, however, are set in a finite state space $X$ consisting only of three elements. \\
The following discussion concerns a lift of Corollary \ref{cor:OptPeakGen} to a given $\ell'>3$. So now the focus lies on the question whether or not the suitably adapted statement of Corollary \ref{cor:OptPeakGen} holds for any number of marginals $N\geq 2$ when paired with $\ell'$. Note that in the introduction specific 'pairable' $N$'s for any $\ell'>3$ are given.\\
Already when increasing the number of states to $\ell=4$ and keeping the number of marginals at $N=2$ the representing measure of the 'peak' that is $ER^{(N)}$ blossoms into multiple extreme points of the symmetric Kantorovich polytope for $2$ marginals and $4$ states. That is, there exist multiple extreme points of the symmetric Kantorovich polytope for $2$ marginals and $4$ states that are in line with the support-restrictions provided by Theorem \ref{the:SupOptC}. Hence, the proof of Corollary \ref{cor:OptPeakGen} can not be lifted to the case of $\ell=4$ states. However, we suspect that each one of the 'in line'-extreme points of the symmetric Kantorovich polytope is of Monge-type for $\ell=4$ as well as $\ell=5$ states and any number of marginals $N\geq 2$. So even though one is not able to identify a specific extreme point as unique optimizer we believe that the support-restriction suffices to at least classify the optimizer(s) as Monge. When increasing the number of states to $\ell=6$ this door closes as 
\begin{equation}
\label{eq:NonMExt}
\frac{1}{4} S \delta_{123} + \frac{1}{4} S \delta_{145} + \frac{1}{4} S \delta_{246} + \frac{1}{4} S \delta_{356}
\end{equation}
is an 'in-line' extreme point of the symmetric Kantorovich polytope for $N=3$ marginals and $\ell=6$ states that is not of Monge-type.
\\
An interesting question to pose is now whether or not non-Monge states of form \eqref{eq:NonMExt} remain extremal when moving from $N$-point to two-point costs. A more general question - we believe - this section accumulates to is the following. Given a symmetric pair-cost $v$ that fulfills the condition on its diagonal stated in Theorem \ref{the:SupOptC}, what geometric attributes of the off-diagonal part of $v$ decide whether the optimizer is of Monge-type or not - provided it exists a unique optimizer.

\bibliographystyle{plain}

\bibliography{bibli}

\begin{thebibliography}{10}

\bibitem{AC11}
M.~Agueh and G.~Carlier.
\newblock Barycenters in the {W}asserstein {S}pace.
\newblock {\em SIAM J. Math. Anal.}, 43(2):904--924, 2011.

\bibitem{BHP13}
M.~Beiglb\"ock, P.~Henry-Labord\`ere, and F.~Penkner.
\newblock Model-independent bounds for option prices - a mass transport
  approach.
\newblock {\em Finance Stoch.}, 17(3):477--501, 7 2013.

\bibitem{Be09}
D.~P. Bertsekas.
\newblock {\em Convex {O}ptimization {T}heory}.
\newblock Athena Scientific, Belmont, MA, 1 edition, 2009.

\bibitem{Bi46}
G.~Birkhoff.
\newblock Tres observaciones sobre el algebra lineal.
\newblock {\em Universidad Nacional de Tucuman Revista, Serie A},
  5(3):147--151, 1946.

\bibitem{BDM12}
R.~Burkard, M.~Dell'Amico, and S.~Martello.
\newblock {\em Assignment {P}roblems. {R}evised reprint.}
\newblock Society for Industrial and Applied Mathematics, 2012.

\bibitem{BDG12}
G.~Buttazzo, L.~De~Pascale, and P.~Gori-Giorgi.
\newblock Optimal-transport formulation of electronic density-functional
  theory.
\newblock {\em Phys. Rev. A}, 85:062502, 6 2012.

\bibitem{Ca03}
G.~Carlier.
\newblock On a {C}lass of {M}ultidimensional {O}ptimal {T}ransportation
  {P}roblems.
\newblock {\em J. Convex Anal.}, 10(2):517--529, 02 2003.

\bibitem{CE10}
G.~Carlier and I.~Ekeland.
\newblock Matching for {T}eams.
\newblock {\em Econom. Theory}, 42(2):397--418, 2 2010.

\bibitem{CN08}
G.~Carlier and B.~Nazaret.
\newblock Optimal transportation for the determinant.
\newblock {\em ESAIM Control Optim. Calc. Var.}, 14(4):678--698, 2008.

\bibitem{CFM14}
H.~Chen, G.~Friesecke, and C.~B. Mendl.
\newblock Numerical {M}ethods for a {K}ohn-{S}ham {D}ensity {F}unctional
  {M}odel {B}ased on {O}ptimal {T}ransport.
\newblock {\em J. Chem. Theory Comput.}, 10(10):4360--4368, 2014.
\newblock PMID: 26588133.

\bibitem{CMN10}
P.-A. Chiappori, R.~J. McCann, and L.~P. Nesheim.
\newblock Hedonic price equilibria, stable matching, and optimal transport:
  equivalence, topology, and uniqueness.
\newblock {\em Econom. Theory}, 42(2):317--354, Feb 2010.

\bibitem{CDD13}
M.~Colombo, L.~De~Pascale, and S.~Di~Marino.
\newblock Multimarginal {O}ptimal {T}ransport {M}aps for {O}ne-dimensional
  {R}epulsive {C}osts.
\newblock {\em Canad. J. Math.}, 67:350--368, 2013.

\bibitem{CFK13}
C.~Cotar, G.~Friesecke, and C.~Kl{\"u}ppelberg.
\newblock Density {F}unctional {T}heory and {O}ptimal {T}ransportation with
  {C}oulomb {C}ost.
\newblock {\em Comm. Pure Appl. Math.}, 66(4):548--599, 2013.

\bibitem{CFP15}
C.~Cotar, G.~Friesecke, and B.~Pass.
\newblock Infinite-body optimal transport with {C}oulomb cost.
\newblock {\em Calc. Var. Partial Differential Equations}, 54(1):717--742, 9
  2015.

\bibitem{Cs70}
J.~Csima.
\newblock Multidimensional {S}tochastic {M}atrices and {P}atterns.
\newblock {\em J. Algebra}, 14(2):194 -- 202, 1970.

\bibitem{GF18}
G.~Friesecke.
\newblock A simple counterexample to the {M}onge ansatz in multi-marginal
  optimal transport, convex geometry of the set of {K}antorovich plans, and the
  {F}renkel-{K}ontorova model.
\newblock {\em ArXiv e-prints}, 8 2018.

\bibitem{FMPCK13}
G.~Friesecke, C.~B. Mendl, B.~Pass, C.~Cotar, and C.~Kl{\"u}ppelberg.
\newblock N-density representability and the optimal transport limit of the
  {H}ohenberg-{K}ohn functional.
\newblock {\em J. of Chem. Phys.}, 139(16):164109, 2013.

\bibitem{FV18}
G.~Friesecke and D.~V\"ogler.
\newblock Breaking the {C}urse of {D}imension in {M}ulti-{M}arginal
  {K}antorovich {O}ptimal {T}ransport on {F}inite {S}tate {S}paces.
\newblock {\em SIAM J. Math. Anal.}, 50(4):3996--4019, 2018.

\bibitem{GHT14}
A.~Galichon, P.~Henry-Labordère, and N.~Touzi.
\newblock A stochastic control approach to no-arbitrage bounds given marginals,
  with an application to lookback options.
\newblock {\em Ann. Appl. Probab.}, 24(1):312--336, 2014.

\bibitem{GS98}
W.~Gangbo and A.~\'Swi{\k{e}}ch.
\newblock Optimal maps for the multidimensional {M}onge-{K}antorovich problem.
\newblock {\em Comm. Pure Appl. Math.}, 51(1):23--45, 1998.

\bibitem{GKR18}
A.~Gerolin, A.~Kausamo, and T.~Rajala.
\newblock Non-existence of optimal transport maps for the multi-marginal
  repulsive harmonic cost.
\newblock {\em ArXiv e-prints}, 5 2018.

\bibitem{He02}
H.~Heinich.
\newblock Probl{\'e}me de {M}onge pour probabilit{\'e}s.
\newblock {\em C. R. Math. Acad. Sci. Paris}, 334(9):793--795, 12 2002.

\bibitem{Ho94}
L.~Hoermander.
\newblock {\em Notions of {C}onvexity}.
\newblock Birkh\"auser, Boston, Massachusetts, 1994.

\bibitem{MATLAB2018b}
The~MathWorks Inc.
\newblock {\em MATLAB and Statistics Toolbox R2018b (MATLAB 9.5)}.
\newblock Natick, Massachusetts, United States, 2018.

\bibitem{Kr07}
V.~M. Kravtsov.
\newblock Combinatorial properties of noninteger vertices of a polytope in a
  three-index axial assignment problem.
\newblock {\em Cybernet. Systems Anal.}, 43(1):25--33, Jan 2007.

\bibitem{LL14}
N.~Linial and Z.~Luria.
\newblock On the {V}ertices of the d-{D}imensional {B}irkhoff {P}olytope.
\newblock {\em Discrete Comput. Geom.}, 51(1):161--170, Jan 2014.

\bibitem{MP17}
A.~Moameni and B.~Pass.
\newblock Solutions to multi-marginal optimal transport problems concentrated
  on several graphs.
\newblock {\em ESAIM Control Optim. Calc. Var.}, 23(2):551--567, 2017.

\bibitem{Pa11}
B.~Pass.
\newblock Uniqueness and {M}onge {S}olutions in the {M}ultimarginal {O}ptimal
  {T}ransportation {P}roblem.
\newblock {\em SIAM J. Math. Anal.}, 43(6):2758--2775, 2011.

\bibitem{Pa12}
B.~Pass.
\newblock On the local structure of optimal measures in the multi-marginal
  optimal transportation problem.
\newblock {\em Calc. Var. Partial Differential Equations}, 43(3):529--536, Mar
  2012.

\bibitem{Pa13}
B.~Pass.
\newblock Remarks on the semi-classical {H}ohenberg-{K}ohn functional.
\newblock {\em Nonlinearity}, 26(9):2731, 2013.

\bibitem{POLYMAKE32}
The polymake team.
\newblock {\em polymake 3.2}.
\newblock Discrete Mathematics/Geometry Institut für Mathematik der Technischen
  Universit\"at Berlin, Straße des 17. Juni 136, Berlin, Deutschland, 2018.

\bibitem{RPDB12}
J.~Rabin, G.~Peyr{\'e}, J.~Delon, and M.~Bernot.
\newblock Wasserstein {B}arycenter and {I}ts {A}pplication to {T}exture
  {M}ixing.
\newblock {\em Scale Space and Variational Methods in Computer Vision}, pages
  435--446, 2012.

\bibitem{Ro97}
R.~T. Rockafellar.
\newblock {\em Convex {A}nalysis}.
\newblock Princeton University Press, New Jersey, 1997.

\bibitem{Sp00}
F.~C.~R. Spieksma.
\newblock {\em Multi {I}ndex {A}ssignment {P}roblems: {C}omplexity,
  {A}pproximation, {A}pplications}, pages 1--12.
\newblock Springer US, Boston, MA, 2000.

\bibitem{Vi09}
C.~Villani.
\newblock {\em Optimal {T}ransport: {O}ld and {N}ew}.
\newblock Springer Verlag, Berlin Heidelberg, 2009.

\bibitem{vN53}
J.~von Neumann.
\newblock A certain zero-sum two person game equivalent to the optimal
  assignment problem.
\newblock {\em Contributions to the Theory of Games}, 11:5--12, 1953.

\end{thebibliography}
\end{document}